\documentclass[12pt]{article}
\usepackage[margin=1in]{geometry}
\usepackage{setspace}

\usepackage{amsmath,amsthm,amssymb}
\usepackage{bbm}
\usepackage{color,colortbl}
\usepackage{subcaption}
\usepackage{multirow,arydshln}
\allowdisplaybreaks

\usepackage{tikz}
\usetikzlibrary{arrows,automata,positioning}
\usepackage{pgfplots}
\pgfplotsset{compat=newest}

\newtheorem{lemma}{Lemma}[section]
\newtheorem{proposition}[lemma]{Proposition}
\newtheorem{theorem}[lemma]{Theorem}
\newtheorem{corollary}[lemma]{Corollary}
\newtheorem{definition}[lemma]{Definition}
\newtheorem{example1}[lemma]{Example}
\newtheorem{rem1}[lemma]{Remark}
\newtheorem{assumption}[lemma]{Assumption}
\newtheorem{alg1}[lemma]{Algorithm}
\newtheorem{me1}[lemma]{Mechanism}
\newenvironment{remark}{\begin{rem1}\rm}{\end{rem1}}
\newenvironment{example}{\begin{example1}\rm}{\end{example1}}

\newcommand{\bbr}{\mathbb{R}}
\newcommand{\R}{\mathbb{R}}
\newcommand{\bbx}{\mathbb{X}}

\newcommand{\xcal}{\mathcal{X}}
\newcommand{\pcal}{\mathcal{P}}
\newcommand{\wcal}{\mathcal{W}}
\newcommand{\tcal}{\mathcal{T}}
\newcommand{\T}{\top}
\newcommand{\diag}{\operatorname{diag}}
\newcommand{\ind}{\mathbbm{1}}
\newcommand{\cl}{\operatorname{cl}}
\newcommand{\co}{\operatorname{co}}
\newcommand{\bd}{\operatorname{bd}}
\renewcommand{\int}{\operatorname{int}}

\DeclareMathOperator*{\argmin}{arg\,min}

\begin{document}

\title{Deep Learning the Efficient Frontier of Convex Vector Optimization Problems} 
\author{Zachary Feinstein \thanks{Stevens Institute of Technology, School of Business, Hoboken, NJ 07030, USA, zfeinste@stevens.edu.} \and Birgit Rudloff \thanks{Vienna University of Economics and Business, Institute for Statistics and Mathematics, Vienna A-1020, AUT, brudloff@wu.ac.at.}}
\maketitle
\abstract{
In this paper, we design a neural network architecture to approximate the weakly efficient frontier of convex vector optimization problems (CVOP) satisfying Slater's condition.  The proposed machine learning methodology provides both an inner and outer approximation of the weakly efficient frontier, as well as an upper bound to the error at each approximated efficient point.  In numerical case studies we demonstrate that the proposed algorithm is effectively able to approximate the true weakly efficient frontier of CVOPs. This remains true even for large problems (i.e., many objectives, variables, and constraints) and thus overcoming the curse of dimensionality.
}

\section{Introduction}\label{sec:intro}

Vector and multiobjective optimization is the field of maximizing or minimizing multiple objective functions simultaneously.  As with traditional (scalar) optimization, these problems can be classified as, e.g., linear, convex, or nonconvex.  Within this work we focus on convex vector optimization problems (CVOP).  Such problems are widely applicable in practice.  Specifically, multiobjective optimization has been applied to fields as diverse as finance (for the mean-risk problem of~\cite{markowitz1952} in~\cite{KR18} as well as for set-valued risk measures~\cite{FR14-alg}), electric and power systems \cite{cassidy2016risk}, structural design \cite{marler2004survey} and references therein, and chemical engineering \cite{bhaskar2000applications}.

Convex vector optimization problems are those with two or more convex objectives and convex constraints.  Already multiple algorithms exist to approximate the solutions for these problems (or for specific subclasses of such problems).  We refer the interested reader to \cite{RW05}
for an overview of some algorithms for CVOPs.  For numerical tractability, these methods result in (polyhedral) approximations of the efficient frontier.  That is, rather than finding the exact solution, only an (inner and/or outer) approximation of the efficient frontier can be found.  We wish to highlight algorithms based on Benson's approximation algorithm~\cite{benson98} for linear vector optimization problems, that have been generalized to the convex case in~\cite{ESS11, LRU14, DLSW21, KU22}. 
While powerful, these algorithms suffer from curse of dimensionality as the computational complexity of such problems grows exponentially in the dimension of the problem as, e.g., the vertex enumeration subproblem is NP-hard~\cite{khachiyan2008generating}. Typically in practice, when the number of objective functions is five or higher, those algorithms cannot solve the vector optimization problem anymore; this computational bound holds for convex or even linear vector optimization problems. Such computational constraints similarly hold true for other types of algorithms, e.g., the box coverage algorithm of~\cite{eichfelder2021approximation} for computing approximate solutions of multiobjective optimization problems. In Section~\ref{sec:many} we will revisit a test instance from~\cite{eichfelder2021approximation} that could not be solved (in a reasonable amount of time or for higher dimensions at all) for dimension four and higher. In contrast, with the machine learning methodology that we propose in this paper, we can still solve that problem for a $20$ dimensional objective space in under a minute on a personal machine. 
Similarly, the sandwich algorithm from~\cite{bokrantz2013algorithm} provides an inner and outer approximation of the efficient frontier, but suffers also the curse of dimensionality as it involves vertex enumeration. In~\cite{bokrantz2013algorithm} problems up to $14$ objectives have been solved. We will reuse example~7.2 from~\cite{bokrantz2013algorithm} and solve it up to $5000$ objectives instead.

Within this work we construct a new approximation method for CVOPs using neural networks.  The primary contributions and innovations of this work are two-fold.
\begin{enumerate}
\item We propose a machine learning methodology to (approximately) solve CVOPs by considering the weighted-sum scalarizations of the multiobjective problem.  Due to the use of neural networks, the proposed algorithm can tractably be applied to large vector optimization problems without the need for vertex enumeration (as required by, e.g., \cite{benson98,bokrantz2013algorithm,LRU14}).  That is, we are able to overcome the curse of dimensionality in solving vector optimization problems. 
In fact, as demonstrated in numerical case studies, the computational time required for the machine learning approach undertaken herein grows (roughly) linearly with the problem size.
\item By proposing two neural networks that are jointly trained, the proposed machine learning methodology provides both inner and outer approximations of the entire weakly efficient frontier of the relevant CVOP.  In constructing the algorithm in this way, we are able to quantify (an upper bound to) the error of these approximations at each point on the weakly efficient frontier.  By having inner and outer approximations with a known error, this machine learning methodology can be applied in practice without the typical risks of machine learning -- the uncertain approximation error. In comparison to a naive application of machine learning for vector optimization problems, these neural networks were motivated by the vertex enumeration approaches (e.g., \cite{bokrantz2013algorithm,LRU14,eichfelder2021approximation}) which also provide inner and outer approximations. 
\end{enumerate}

The organization of this paper is as follows.  
In Section~\ref{sec:background}, the background, notation, and structure of both CVOPs and neural networks is provided.
In Section~\ref{sec:nn}, the proposed neural network architecture for approximating the weakly efficient frontier of a \emph{strictly} convex vector optimization problem is provided.  Specifically, a primal neural network is proposed in Section~\ref{sec:nn-primal}, a dual neural network is provided in Section~\ref{sec:nn-dual}, and the joint loss function for these neural networks is given in Section~\ref{sec:nn-loss}. Then, in Section~\ref{sec:approx}, inner and outer approximations of the upper image of the problem are constructed from these neural networks.  Within this section, we additionally provide a function which provides an upper bound to the errors for the inner approximation at each point on the weakly efficient frontier.  We implement the proposed machine learning approach for four strictly convex optimization problems in Section~\ref{sec:cs}.  We use these problems both to validate the proposed approach and to demonstrate its applicability to large-scale problems thus overcoming the curse of dimensionality.
Finally, in Section~\ref{sec:convex}, we extend the machine learning framework of Section~\ref{sec:nn} to allow for convex, but not strictly convex, vector optimization problems.  This is demonstrated on two small problems -- a linear vector optimization problem and the mean-risk portfolio optimization problem -- that permit simple visualizations.

\section{Background}\label{sec:background}
Within this section, we provide the background material necessary for this work.  First, in Section~\ref{sec:background-vop}, the vector optimization setting that is considered throughout this work is presented along with basic definitions and results.  Second, in Section~\ref{sec:background-nn}, some background on neural networks and their usage for approximating, i.e.\ learning, a target function, are given.

\subsection{Vector optimization problems}\label{sec:background-vop}

Consider a decision maker attempting to simultaneously minimize $P$ objective functions $f_i: \bbx \to \R$ for $i=1,...,P$ over the set of decisions $\bbx \subseteq \R^N$.    Throughout, we assume that the decision space $\bbx \neq \emptyset$ is nonempty and that the ordering cone in the objective space is $\R^P_+$.  To simplify notation, we denote $f(x) := (f_1(x),...,f_P(x))^\T$ for any $x \in \bbx$.  That is, throughout this work we consider a vector optimization problem (VOP) of the form
\begin{equation}\label{eq:vop}
\min\{f(x) \; | \; x \in \bbx\}.
\end{equation}

Before continuing to the main results of this work, we provide some basic definitions and results on vector optimization.  We refer the interested reader to, e.g.,~\cite{jahn2011vector} for more details on these problems and definitions. We denote by $f[\bbx]=  \{f(x) \; | \; x \in \bbx\}$ the image of the feasible set.
\begin{definition}\label{defn:bounded}
The VOP \eqref{eq:vop} is \textbf{\emph{bounded}} if there exists some $y \in \R^P$ such that $f[\bbx] \subseteq \{y\} + \R^P_+$.
\end{definition}

\begin{definition}\label{defn:pareto}
Consider the VOP \eqref{eq:vop}.
$x^* \in \bbx$ is a \textbf{\emph{(weak) minimizer}} if $f(x^*) - f(x) \in \R^P_+$ for some $x \in \bbx$ implies $f(x) = f(x^*)$ (resp.\ $f(x^*) - f(x) \not\in \R^P_{++}$ for every $x \in \bbx$).
The set of (weak) minimizers is denoted by $\xcal^* \subseteq \bbx$ (resp.\ $\xcal \subseteq \bbx$).
The \textbf{\emph{(weak) efficient frontier}} is given by $f[\xcal^*] := \{f(x^*) \; | \; x^* \in \xcal^*\}$ (resp.\ $f[\xcal]$). 
The \textbf{\emph{upper image}} is provided by $\pcal := \cl\left(f[\xcal] + \R^P_+\right)$.
\end{definition}

Within this work we are specifically interested in problems with feasible space
\begin{equation}\label{eq:constraints}
\bbx := \{x \in \R^N \; | \; g(x) \in -\R^M_+\}
\end{equation}
for $g: \R^N \to \R^M$.  
In particular, we are interested in \emph{convex} vector optimization problems (CVOP), i.e., problems of the form
\begin{equation}\label{eq:cvop}\tag{P}
\min\{f(x) \; | \; g_j(x) \leq 0,~j = 1,...,M\}
\end{equation}
such that $f_i,g_j$ are convex functions for every $i = 1,...,P$ and $j = 1,...,M$.

It is well known that (weak) minimizers of~\eqref{eq:cvop} are related to solutions of the weighted-sum scalarization problem.
\begin{proposition}\label{prop:scalar}
Consider the CVOP~\eqref{eq:cvop}.
\begin{enumerate}
\item $x^*$ is a weak minimizer if and only if there exists some $w \in \wcal := \{w \in \R^P_+ \; | \; {\bf 1}^\T w = 1\}$, where we denote ${\bf 1} := (1,1,...,1)^\T \in \R^P$, such that $x^*$ solves the weighted-sum scalarization problem
    \begin{equation}\label{eq:scalarization}\tag{wP}
    \min\Big\{\sum_{i = 1}^P w_i f_i(x) \; | \; g_j(x) \leq 0,~j = 1,...,M\Big\}.
    \end{equation}
\item $x^*$ is a minimizer if it solves the weighted-sum scalarization problem~\eqref{eq:scalarization} for some $w \in \wcal \cap \R^P_{++}$.
\end{enumerate}
\end{proposition}
\begin{proof}
The first statement follows from Corollary 5.29 of~\cite{jahn2011vector}.  The second statement follows from Theorem 5.18(b) of~\cite{jahn2011vector}.
\end{proof}

\begin{assumption}\label{ass:slater}
Throughout this work we will make the following assumptions:
\begin{itemize}
\item $f_i$ are strictly convex and continuously differentiable;
\item $g_j$ are convex and continuously differentiable;
\item \eqref{eq:cvop} is bounded;
\item Slater's condition holds, i.e., there exists some $\bar x \in \R^N$ such that $g_j(\bar x) < 0$ for every $j = 1,...,M$. 
\end{itemize}
\end{assumption}

\begin{remark}
\label{rem:convex}
Let us remark on some possible relaxations of these assumptions.
\begin{enumerate}
\item\label{rem:convex-convex}
Note that the strict convexity of the objective functions $f_i$ as encoded in Assumption~\ref{ass:slater} could be relaxed for all results presented within this work so long as the weighted-sum scalarization problem~\eqref{eq:scalarization} is strictly convex for all scalarizations $w \in \wcal$ of interest. This will be used in the example presented in Section~\ref{sec:mean-var}. It will also prove useful in Section~\ref{sec:convex} to allow for general convex, but not strictly convex, objective functions $f_i$ by augmenting the problem with an additional strictly convex objective function. This augmentation ensures that the new objective has strictly convex scalarizations for all $w \in \wcal$ of interest.
\item\label{rem:convex-equality}
As presented above, we focus solely on inequality constrained CVOPs within this work.  
We wish to note that linear equality constraints can be introduced within this setting as well.
Consider the CVOP
\begin{equation*}
\min\{f(\hat{x}) \; | \; g_j(\hat{x}) \leq 0,~j = 1,...,M,\; A\hat{x} = b\}
\end{equation*}
for $A \in \R^{O \times N}$ and $b \in \R^O$.  Without loss of generality, we will assume $A$ is of full rank.
Let $A_\perp \in \R^{N \times (N-O)}$ provide a basis for the null space of $A$; that is, $A A_\perp z = {\bf 0}$ for any $z \in \R^{N-O}$.
Furthermore, let $\tilde{x} \in \R^N$ define a particular solution of the linear constraints, i.e, $A\tilde{x} = b$.
The equality constrained CVOP can then be reformulated in the form~\eqref{eq:cvop} as
\begin{equation*}
\min\{f(A_{\perp}x + \tilde{x}) \; | \; g_j(A_{\perp}x + \tilde{x}) \leq 0,~j = 1,...,M\} 
\end{equation*}
with $N-O$ variables.  Provided this reformulated problem satisfies Assumption~\ref{ass:slater}, the methodology considered within this work can be applied directly to this reformulation. This will be used in the examples considered in Sections~\ref{sec:mean-var} and~\ref{sec:convex}.
\item
Herein we take the ordering cone $\R^P_+$.  All results can be considered in which the minimization is taken with respect to an ordering cone $C \subseteq \R^p$ with nonempty interior instead.  The only modification necessary is that the scalarizations are taken with respect to $w \in C^+ := \{w \in \R^p \; | \; c^\T w \geq 0 \; \forall c \in C\}$ such that $w^\T c = 1$ for some fixed element $c \in \int C$ provided the objective $f$ is $C$-convex.
\end{enumerate}
\end{remark}

The weakly efficient frontier can be traced through considerations of $w\in\wcal\mapsto f(x^*(w))$ where $x^*(w)$ is a solution to the scalarization problem~\eqref{eq:scalarization} in Proposition~\ref{prop:scalar} for every $w \in \wcal$. We recall that the goal of this work is to determine (an approximation of) the weakly efficient frontier. As such, we seek to approximate the mapping $w\in\wcal \mapsto x^*(w)$ rather than just finding its value at a finite number of scalarizations.

In addition to the primal scalarization problem~\eqref{eq:scalarization} we will also consider its Lagrange dual problem.  For that, consider the dual function $d: \R^M \times \wcal \to \R$ defined by
\begin{equation}\label{eq:dual}
d(\lambda,w) := \inf_{x \in \R^N} \Big[\sum_{i = 1}^P w_i f_i(x) + \sum_{j = 1}^M \lambda_j g_j(x)\Big].
\end{equation}
The dual problem is then to maximize this dual function subject to the dual feasibility constraints, i.e.,
\begin{equation}\label{eq:scalarization-dual}\tag{wD}
\max \Big\{\inf_{x \in \R^N} \Big[\sum_{i = 1}^P w_i f_i(x) + \sum_{j = 1}^M \lambda_j g_j(x)\Big] \; | \; \lambda \in \R^M_+\Big\}.
\end{equation}
By Assumption~\ref{ass:slater}, strong duality is satisfied for these scalar problems, i.e.\ the optimal values of~\eqref{eq:scalarization} and~\eqref{eq:scalarization-dual} are equal, and a dual solution exists.

The method presented in this paper to approximate the weakly efficient frontier by machine learning methods will rely heavily on the following well-known KKT conditions. Note that the KKT conditions~\eqref{eq:kkt}-\eqref{eq:kkt-cs} below are the usual KKT conditions of the scalarization problem~\eqref{eq:scalarization}, but can also be obtained using the multiobjective KKT conditions for the CVOP~\eqref{eq:cvop} directly (without the need to go through the scalarization first), see~\cite{eichfelder2021proximity}.

\begin{theorem}\label{thm:kkt}
$(x^*(w),\lambda^*(w)) \in \R^N \times \R^M$ are a pair of primal ($x^*(w)$) and dual ($\lambda^*(w)$) solutions of the scalarization problems~\eqref{eq:scalarization}, respectively~\eqref{eq:scalarization-dual}, w.r.t.\ $w \in \wcal$ if and only if they jointly satisfy the following KKT conditions
\begin{align}
\label{eq:kkt}  {\bf 0} &= Jf(x^*(w))^\T w + Jg(x^*(w))^\T \lambda^*(w) \\
\label{eq:kkt-primal}  {\bf 0} &\geq g(x^*(w)) \\
\label{eq:kkt-dual}  {\bf 0} &\leq \lambda^*(w) \\
\label{eq:kkt-cs}  0 &= \lambda_i^*(w) g_i(x^*(w))  \quad i=1,...,m,
\end{align}
where $J$ denotes the Jacobian of the corresponding function.
\end{theorem}
\begin{proof}
Due to the assumption that Slater's condition holds (Assumption~\ref{ass:slater}), this result can be found in, e.g.,~\cite[Chapter 5.5.3]{boyd2004convex}.
\end{proof}

\subsection{Neural networks}\label{sec:background-nn}
Consider, now, the task of approximating a function $y^*: \tcal \to \R^m$ for some \emph{compact} set of input variables $\tcal \subseteq \R^n$.  Such an approximation problem is, fundamentally, a regression problem.  Within this work, we will focus on feed-forward neural networks to regress the function $y^*$. These functions can be viewed as a multi-stage regression model.  We refer the interested reader to, e.g.,~\cite[Chapter 6]{goodfellow2016deep} for more details on feed-forward neural networks.
\begin{definition}\label{defn:nn}
An $\ell$-hidden layer \textbf{\emph{neural network}} with $h_l$ nodes in layer $l \in \{0,1,...,\ell+1\}$ (such that $h_0 := n$ and $h_{\ell+1} := m$ for the input and output layers respectively) is a mapping $y_{\ell+1}: \tcal \times \Theta \to \R^m$ that can be decomposed as
\begin{align*}
y_l(t,\theta) &:= \Phi_l(\theta_{l,b} + \theta_{l,w} y_{l-1}(t,\theta)) \quad \forall l \in \{1,...,\ell+1\} 
\end{align*}
with $y_0(t,\theta) := t$
for \textbf{\emph{activation functions}} $\Phi_l: \R^{h_l} \to \R^{h_l}$ for every $l = 1,...,\ell+1$ and parameter space $\Theta \subseteq \{(\theta_{l,b},\theta_{l,w})_{l = 1}^{\ell+1} \; | \; \theta_{l,b} \in \R^{h_l}, \; \theta_{l,w} \in \R^{h_l \times h_{l-1}} \; \forall l = 1,...,\ell+1\}$.  The parameter $\theta_{l,b}$ is often called the \textbf{\emph{bias}} and $\theta_{l,w}$ is often called the \textbf{\emph{weights}} of the $l^{th}$ hidden layer.
Jointly, the number of hidden layers $\ell$, the number of nodes within each hidden layer $h_l$ for $l \in \{1,...,\ell\}$, the activation functions $\Phi_l$ for $l \in \{1,...,\ell+1\}$, and the restrictions on the parameter space $\Theta$ are called the \textbf{\emph{hyperparameters}} of the neural network.
\end{definition}
\begin{remark}\label{rem:density}
The connection structure of a neural network in layer $l$ is defined through the sparsity structure of the weights $\theta_{l,w}$ imposed within the parameter space $\Theta$.  A dense neural network, which we consider throughout the remainder of this work, is one in which $\Theta = \{(\theta_{l,b},\theta_{l,w})_{l = 1}^{\ell+1} \; | \; \theta_{l,b} \in \R^{h_l}, \; \theta_{l,w} \in \R^{h_l \times h_{l-1}} \; \forall l = 1,...,\ell+1\}$.
\end{remark}
Before continuing, we wish to consider some common choices of activation functions used within the neural network literature.
\begin{example}\label{ex:activation}
Each of the following activation functions $\phi: \R \to \R$ are often applied component-wise in practice (i.e., such that $\Phi( \textbf{z}) = (\phi(z_1),...,\phi(z_k))^\T$ for $\textbf{z} \in \R^k$).  Let $z \in \R$ be arbitrary. 
\begin{enumerate}
\item \textbf{Linear}: Consider the identity mapping $\phi(z) := z$.  This activation function is often denoted as the linear activation function as applying it along with the bias and weights creates a linear regression within that layer of the neural network.  Because the composition of linear functions is again linear, this activation function is typically only used for the output layer $\ell+1$ in practice.
\item \textbf{Rectified linear unit [ReLU]}: Consider the positive mapping $\phi(z) := z^+$.  This activation function takes just the positive part of the input, setting all negative values to $0$.
\item \textbf{Smooth ReLU}: Consider the mapping $\phi(z) := \log(1+\exp(z))$.  By construction, this activation function provides a positive value that limits to the identity mapping for positive inputs and 0 for negative inputs.  However, unlike the ReLU activation function, this activation is smooth. 
\item \textbf{Hyperbolic tangent}: Consider the mapping $\phi(z) := \operatorname{tanh}(z)$.  This activation function provides an S-shaped curve that tempers extreme values as it limits to plus or negative 1 for positive of negative inputs respectively.
\end{enumerate}
\end{example}

One fundamental reason for the prevalence of neural networks within the machine learning community is the so-called \emph{universal approximation theorem}.  This result says that for any \emph{continuous} target function $y^*: \tcal \to \R^m$, there exists a \emph{single-layer} ($\ell = 1$) neural network $y: \tcal \to \R^m$ that can uniformly approximate $y^*$.  This result has been expanded to deep neural networks ($\ell \geq 2$) in, e.g.,~\cite{hornik1989multilayer,kidger2020universal}.
\begin{theorem}\label{thm:uat}\cite[Theorem 3.1]{pinkus1999approximation}
Let $y^*: \tcal \to \R^m$ be a continuous function with compact domain and $\phi: \R \to \R$ be a continuous activation function (applied component-wise to vectors as in Example~\ref{ex:activation}).  For every $\epsilon > 0$ there exists a 1-layer neural network $y: \tcal \times \Theta \to \R^m$ with $h_1$ hidden nodes such that 
\[\inf_{\theta \in \Theta} \sup_{t \in \tcal} \|y^*(t) - y(t,\theta)\|_{2} < \epsilon\]
if and only if $\phi$ is not a polynomial.
\end{theorem}

As noted within the statement of the universal approximation theorem (Theorem~\ref{thm:uat}), there only exists parameters $\theta \in \Theta$ such that this approximation holds.  To attempt to find these parameters, we undertake a \emph{training} regiment.
This is accomplished by minimizing some \textbf{loss function} $\bar L: \tcal \times \R^m \times \R^m \to \R_+$ such that $\bar L(t,y,y) = 0$ for any $t \in \tcal$ and $y \in \R^m$ (see, e.g., Definition 3.1 of \cite{scholkopf2002learning}), i.e., we seek the parameters $\theta^* \in \Theta$ such that
\begin{align}
\label{eq:MLopt}
\theta^* \in \argmin_{\theta \in \Theta} \mathbb{E}[\bar L(t,y^*(t),y(t,\theta))]
\end{align}
where the expectation is taken over the input $t \in \tcal$ (jointly with $y^*(t)$ is the true output includes noise). The expected loss $y(\cdot,\theta) \mapsto \mathbb{E}[\bar L(t,y^*(t),y(t,\theta))]$ is often referred to as the \textbf{risk functional} associated with the loss function $\bar L$.
(The loss $\bar L(t,y^*(t),y(t,\theta))$ can be thought of as a distance the neural network $y(t,\theta)$ is from its target value $y^*(t)$. In particular, we use the stronger convention that $\bar L(t,y^*(t),y(t,\theta)) = 0$ if and only if $y(t,\theta) = y^*(t)$.)
However, this minimizing argument need not exist in general and, furthermore, this may be a nonconvex optimization problem; we remark on these issues further within Remark~\ref{rem:training} below.
This process of minimizing the expected loss function is often called \textbf{training} the neural network and it is how the network \textbf{learns}.
Typically the risk functional is evaluated at some \emph{finite} \textbf{training points} $\{t_1,...,t_K\} \subseteq \tcal$ so that it is defined via the empirical measure, i.e., $L(y(\cdot,\theta)) := \frac{1}{K} \sum_{k = 1}^K \bar L(t_k , y^*(t_k) , y(t_k,\theta))$ for the loss function $\bar L$.

\begin{remark}\label{rem:training}
As a note of caution, the universal approximation theorem only guarantees there exists some neural network which can uniformly approximate the target.  However, there is no guarantee that we actually constructed that neural network in practice when setting hyperparameters and training the parameters.  In particular, optimizing the parameters $\theta$ by minimizing the risk functional often requires solving a highly nonlinear optimization problems. Though this problem is typically high-dimensional and nonlinear, multiple algorithms exist to efficiently \emph{approximate} $\theta^*$.  Within this work, we use the Adam optimizer~\cite{kingma2014adam} which is a gradient descent-based approach developed for large-scale problems. As a brief preview for the results of this work, we are essentially substituting the hard problem of convex vector optimization (Problem~\eqref{eq:cvop}) with the previous scalar optimization approach (Problem~\eqref{eq:MLopt}) encoded within these neural network optimization methodologies. 
\end{remark}

For the remainder of this work, to ease notation, we will drop the dependence of neural networks on their parameters $\theta$.

\section{A primal-dual neural network}\label{sec:nn}
Within this section our goal is to construct two neural networks which approximate the primal and dual solutions ($x^*(\cdot),\lambda^*(\cdot)$) to the scalarization problems~\eqref{eq:scalarization} so as to trace approximations of the entire weakly efficient frontier.  These neural networks will be constructed to guarantee primal and dual feasibility as encoded in \eqref{eq:kkt-primal} and \eqref{eq:kkt-dual}, i.e., a primal feasible neural network $x: \wcal \to \bbx$ and a dual neural network $\lambda: \wcal \to \R^M_+$. In addition, we propose a loss function to \emph{jointly} train these neural networks based on the KKT conditions of Theorem~\ref{thm:kkt}.

\subsection{Primal feasible neural network}\label{sec:nn-primal}
Herein we define a neural network with $P$ inputs (given by $w \in \wcal$) and $N$ outputs that guarantees primal feasibility~\eqref{eq:kkt-primal}.  In order to accomplish this, we want to construct the final layer of the neural network $x:\wcal \to \R^N$ so that the constraints encoded within $\bbx$ are satisfied.

Within the following proposition, we construct a projection mapping which, if used as the final activation function for an arbitrary neural network, guarantees primal feasibility.  As such, we use the general recursive structure of a neural network to guarantee that the primal neural network $x:\wcal \to \R^N$ is feasible for every $w \in \wcal$.
\begin{proposition}\label{prop:primal}
Let $z: \wcal \to \R^N$ define an arbitrary neural network.  Let $x: \wcal \to \R^N$ such that
\begin{align}
\label{eq:x} x(w) &:= (1-t^*(z(w)))z(w) + t^*(z(w))\bar x, \\
\label{eq:t} t^*(z) &:= \max_{j: g_j(z) > 0} \left\{\frac{g_j(z)}{g_j(z) - g_j(\bar x)}\right\},
\end{align}
where $t^*(z) = 0$ if $g(z) \in -\R^M_+$ and $\bar x$ is as defined in Assumption~\ref{ass:slater}.
Then $x(w) \in \bbx$ for every $w \in \wcal$.
\end{proposition}
\begin{proof}
Recall by construction of the feasible set that $x(w) \in \bbx$ if and only if $g_j(x(w)) \leq 0$ for every $j = 1,...,M$.  
Consider the $j^{th}$ inequality constraint.  Note that $t^*(z) \in [0,1)$ for any $z$ by construction of $\bar x$ strictly feasible.  Therefore, by convexity and the definition of $t^*$,
\begin{align*}
g_j(x(w)) &= g_j((1-t^*(z(w))) z(w) + t^*(z(w))\bar x)\\
&\leq (1-t^*(z(w))) g_j(z(w)) + t^*(z(w)) g_j(\bar x) \leq 0.
\end{align*}
\end{proof}
We often refer to the final activation function $z \mapsto (1-t^*(z)) z + t^*(z) \bar x$ provided in Proposition~\ref{prop:primal} as a projection operator as it projects the neural network into the feasible region $\bbx$.

\begin{remark}\label{rem:tolerance}
Due to rounding errors in practice, it is recommended to add a tolerance level to the mapping $t^*$ defined in~\eqref{eq:t}.  That is, define the neural network $x: \wcal \to \bbx$ using $t^*(\cdot;\epsilon): \R^N \to [0,1)$ provided by 
\[t^*(z;\epsilon) := \max_{j: g_j(z) \geq -\epsilon} \left\{\frac{g_j(z) + \epsilon}{g_j(z) - g_j(\bar x)}\right\}\]
for fixed tolerance $\epsilon \in (0 , -\max_j g_j(\bar x))$.
\end{remark}

For the remainder of this work we assume the primal neural network is constructed so as to be feasible, i.e., as in Proposition~\ref{prop:primal} with the final activation function of $z \mapsto (1-t^*(z)) z + t^*(z) \bar x$ to a generic neural network $z: \wcal \to \R^N$.
\begin{remark}\label{rem:primal-feasible}
We propose the projection operator as a general formulation for a final activation function which guarantees primal feasibility.  If the feasible set $\bbx$ takes a standard form, then a more direct approach can be utilized instead.  For instance, within Section~\ref{sec:mean-var} below, $\bbx$ is the unit simplex; therein one could instead also guarantee primal feasibility by taking the normalized rectified linear units (i.e., $x_i(w) := z_i(w)^+ / \sum_{j = 1}^N z_i(w)^+$ for every $i$ if $\sum_{j = 1}^N z_i(w)^+ > 0$ and $x(w) = \bar{x}$ otherwise) without needing to directly eliminate the equality constraint through the use of $A_{\perp}$ and changing the input space as described in Remark~\ref{rem:convex}\eqref{rem:convex-equality}.
\end{remark}

\begin{remark}\label{rem:center0}
Due to the initialization and optimization methods typically employed for neural networks (e.g., the Adam optimizer \cite{kingma2014adam}), it can be beneficial to consider the modified CVOP
\[\min\{f(x+\bar x) \; | \; g_j(x+\bar x) \leq 0,~j = 1,...,M\}.\]
For this modified problem, with objective $\hat f(\cdot) := f(\cdot + \bar x)$ and constraints $\hat g(\cdot) := g(\cdot + \bar x)$, we can select the strictly feasible point $\hat{\bar x} := {\bf 0}$.
\end{remark}

\subsection{Dual feasible neural networks}\label{sec:nn-dual}
Herein we define the dual neural network with $P$ inputs (given by $w \in \wcal$) and $M$ outputs which guarantees dual feasibility~\eqref{eq:kkt-dual}.  Similar to the approach taken for the primal neural network, we want to construct the final layer of this neural network so that feasibility is guaranteed.

Specifically, feasibility for the dual variables associated with the inequality constraints requires non-negativity~\eqref{eq:kkt-dual}.  This can be accomplished by simply applying the rectified linear unit as introduced in Example~\ref{ex:activation} (or the smooth approximation of the positive function) as a final activation function to a neural network with $M$ outputs already, i.e., $\lambda(w) := \tilde \lambda(w)^+$ for neural network $\tilde \lambda: \wcal \to \R^M$.

\subsection{Loss function}\label{sec:nn-loss}

In order to train the parameters of our two neural networks, we need to also construct a loss function which we seek to minimize.  In particular, we design a single loss function to jointly train both neural networks simultaneously.  To accomplish this, we will use the KKT conditions presented in Theorem~\ref{thm:kkt}.

Specifically, to find the parameters for our neural networks, we select $K$ (representative) choices of the scalarization weights $w_k \in \wcal$ for $k = 1,...,K$\footnote{As clarified in the subsequent case studies, these scalarization weights are chosen either via a regular grid or sampled uniformly within $\wcal$.} and seek to minimize the deviation our neural networks have from the KKT conditions.  Due to the design of our neural networks $(x(\cdot),\lambda(\cdot))$ as presented above, the primal and dual feasibility conditions~\eqref{eq:kkt-primal} and~\eqref{eq:kkt-dual} are automatically satisfied and, as such, do not need to be considered in our loss function.  Therefore, the deviation from the KKT conditions can be fully characterized by the norms of the right-hand sides of the first order condition~\eqref{eq:kkt} and complimentary slackness condition~\eqref{eq:kkt-cs}. As these errors may be of different orders of magnitude, we choose to weight the error in the complimentary slackness condition by $\eta > 0$.  That is, our empirical risk functional $L$ for jointly training the two neural networks is provided by
\begin{equation}\label{eq:loss}
L(x(\cdot),\lambda(\cdot);\eta) := \frac{1}{K} \sum_{k = 1}^K \left(\begin{array}{l} \|Jf(x(w_k))^\T w_k + Jg(x(w_k))^\T \lambda(w_k)\|_2^2 \\ \quad + \eta \|\diag(\lambda(w_k)) g(x(w_k))\|_2^2\end{array}\right).
\end{equation}
\begin{remark}\label{rem:tolerance-loss}
\begin{enumerate}
\item The loss associated with the complimentary slackness conditions may be modified to account for the rounding errors as discussed in Remark~\ref{rem:tolerance}.  For instance, for tolerance $\epsilon > 0$, we can consider the complimentary slackness error to be $\|\diag(\lambda(w_k)) \diag(\ind_{g_j(x(w_k)) \leq -\epsilon}) g(x(w_k))\|_2^2$ with indicator function $\ind_{(\cdot)}$ taking value 1 if the argument is true and 0 otherwise.  In this way, we treat any solution within $\epsilon$ of the boundary to be sitting on the boundary when considering complimentary slackness.
\item The risk functional $L$ requires an extra hyperparameter $\eta > 0$ to provide a relative weight for the complimentary slackness condition. In practice, we choose 
\[\eta \approx \mathbb{E}[\|JF(x_0(w))^\T w + Jg(x_0(w))^\T \lambda_0(w)\|_2^2]/\mathbb{E}[\|\diag(\lambda_0(w)) g(x_0(w))\|_2^2]\] 
with expectations taken over the scalarizations $w$ and $(x_0(w),\lambda_0(w))$ following the initialization of, e.g., the stochastic gradient descent.
In this way, the initial optimization directions for the neural networks will attempt to minimize both sources of error.
\end{enumerate}
\end{remark}

\begin{remark}\label{rem:dualgap}
We wish to note that the KKT condition based risk functional and loss function~\eqref{eq:loss} introduced above, and used throughout the numerical examples herein, is not the only loss function that can be taken for this problem.  As Assumption~\ref{ass:slater} implies strong duality, the duality gap can be used as a loss function instead via
\[L(x(\cdot),\lambda(\cdot)) := \frac{1}{K} \sum_{k = 1}^K \left[w_k^\T f(x(w_k)) - d(\lambda(w_k),w_k)\right],\]
where $d: \R^M \times \wcal \to \R$ is the dual function provided in~\eqref{eq:dual}.
Notably, this construction removes the direct need to jointly train the primal and dual neural networks as the primal neural network could be trained with risk functional $L_p(x(\cdot)) := \frac{1}{K} \sum_{k = 1}^K w_k^\T f(x(w_k))$ and the dual neural network could be trained with risk functional $L_d(\lambda(\cdot)) := -\frac{1}{K} \sum_{k = 1}^K d(\lambda(w_k),w_k)$ to achieve the same effect.

Though separating the risk functionals so that the primal and dual neural networks can be trained separately is tempting due to the intuitive construction of these objectives, we find that this construction performs worse than the joint KKT based risk functional. Due to the use of the first order condition~\eqref{eq:kkt} in \eqref{eq:loss}, the KKT based risk functional inherently incorporates some elements of sensitivity analysis to improve performance locally around the training data; in contrast, $L_p,L_d$ only seek to optimize at the training data without any additional information to interpolate between the training data.\footnote{This can be seen in, e.g., Figure~\ref{fig:2obj-efficient} below in which we compare the neural network solutions under \eqref{eq:loss} against directly computing the solutions to~\eqref{eq:scalarization} at the training scalarizations.}
\end{remark}

\section{Approximating the weakly efficient frontier}\label{sec:approx}
Within this section we present a method to formally use the neural networks constructed in Section~\ref{sec:nn} to construct inner and outer approximations of the weakly efficient frontier. 

Before formally providing the results on the inner and outer approximations, we wish to recall the (Lagrange) dual function~\eqref{eq:dual} to the scalarization problem~\eqref{eq:scalarization}.  That is, we consider the mapping $d: \R^M \times \wcal \to \R$ defined by
\begin{equation*}
d(\lambda,w) := \inf_{x \in \R^N} \left[w^\T f(x) + \lambda^\T g(x)\right].
\end{equation*}
The dual problem is then to maximize this dual function subject to the dual feasibility constraints~\eqref{eq:kkt-dual} as provided in~\eqref{eq:scalarization-dual}.

The first main result of this section, to prove that the primal neural network can be utilized to provide an inner approximation and the dual neural network can be utilized to provide an outer approximation of the upper set, is formalized in Lemma~\ref{lemma:approx}. Notably, this result is totally reliant on, and trivially follows from, the feasibility of these neural networks as guaranteed by the final activation functions presented in Section~\ref{sec:nn} above.
\begin{lemma}\label{lemma:approx}
The primal neural network $x: \wcal \to \bbx$ provides an inner approximation of the upper set $\pcal$ and the dual neural network $\lambda: \wcal \to \R^M_+$ provide an outer approximation of the upper set $\pcal$ via the relation
\[\cl\co\bigcup_{w \in \wcal}\left[f(x(w)) + \R^P_+\right] \subseteq \pcal \subseteq \bigcap_{w \in \wcal} \left\{y \in \R^P \; | \; w^\T y \geq d(\lambda(w);w)\right\}.\]
\end{lemma}
\begin{proof}
We will prove this result by showing, first, the inner approximation and, second, the outer approximation.
Let $y \in \bigcup_{w \in \wcal} [f(x(w)) + \R^P_+]$.  There exists some $\tilde w \in \wcal$ such that $y - f(x(\tilde w)) \in \R^P_+$.  Therefore, as $x^*(\tilde w)\in\bbx$ by Proposition~\ref{prop:primal}, for any $w \in \wcal$, it holds
\[w^\T y \geq w^\T f(x(\tilde w)) \geq \inf\{w^\T f(x) \; | \; x \in \bbx\}.\]
By Proposition~\ref{prop:scalar}, and noting that the upper set $\pcal$ is closed and convex by assumption, the inner approximation holds.
Now, let $y \in \pcal$.  There exists some $\tilde w \in \wcal$ such that $y - f(x^*(\tilde w)) \in \R^P_+$ for solution $x^*: \wcal \to \bbx$ to \eqref{eq:scalarization}.  Therefore, for any $w \in \wcal$
\begin{align*}
w^\T y &\geq w^\T f(x^*(\tilde w)) \geq \inf\{w^\T f(x) \; | \; x \in \bbx\} \\ 
&= \sup\{d(\lambda,w) \; | \; \lambda \in \R^M_+\} \geq d(\lambda(w),w),
\end{align*}
where the equality holds by strong duality.  
(Though we utilize strong duality here, in fact only weak duality is required.)
As a direct consequence, the outer approximation holds.
\end{proof}

Beyond defining the inner and outer approximations as done in Lemma~\ref{lemma:approx}, we want to also quantify the error from these approximations. To do this, we first wish to define a variable approximation error.  This is accomplished by quantifying the approximation error in any scalarization direction.  As opposed to the standard, constant, approximation error (which we discuss more in Proposition~\ref{prop:defn-approx}) this variable approximation error allows us to more accurately assess the local errors coming from our machine learning approach.  As far as the authors are aware, this definition is novel to the literature.
\begin{definition}\label{defn:approx}
$\tilde\pcal \subseteq \R^P$ is an \textbf{\emph{$\epsilon(\cdot)$-inner approximation}} of~\eqref{eq:cvop} for the function $\epsilon: \wcal \to \R_+$ if
\[\tilde\pcal \subseteq \pcal \subseteq \bigcap_{w \in \wcal} \cl\left[\tilde\pcal - \epsilon(w){\bf 1} + G(w)\right]\]
for the upper set $\pcal$, and where $G(w) := \{y \in \R^P \; | \; w^\T y \geq 0\}$ for any $w \in \wcal$.
\end{definition}
Notably, the quantification of the errors in the $\epsilon(\cdot)$-inner approximation are variable.  In contrast, typically (see, e.g., \cite{LRU14}) a constant error $\epsilon \in \R_{++}$ is taken instead; in such a setting an $\epsilon$-inner approximation $\tilde\pcal \subseteq \R^P$ of $\pcal$ is defined by the relations
\[\tilde\pcal \subseteq \pcal \subseteq \tilde\pcal - \epsilon{\bf 1}.\]
In the following proposition we show that when the error function is constant, Definition~\ref{defn:approx} reduces to the classical definition of an approximation.
\begin{proposition}\label{prop:defn-approx}
Let $\tilde\pcal \subseteq \R^P$ be a closed and convex upper set and let $\epsilon:\wcal \to \R_+$ be a constant function, i.e., $\epsilon(w) \equiv \epsilon_0 > 0$ for every $w \in \wcal$.  $\tilde\pcal$ is an $\epsilon(\cdot)$-inner approximation of $\pcal \subseteq \R^P$ if and only if $\tilde\pcal \subseteq \pcal \subseteq \tilde\pcal - \epsilon_0 {\bf 1}$.
\end{proposition}
\begin{proof}
As $\tilde\pcal$ is a closed and convex upper set, it trivially follows (by the separating hyperplane theorem) that
\[\tilde\pcal - \epsilon_0 {\bf 1} = \bigcap_{w \in \wcal}[\tilde\pcal - \epsilon_0{\bf 1} + G(w)]\]
and the result is proven.
\end{proof}

Within Corollary~\ref{cor:approx} below, we provide an analytical form for the functional approximation error $\epsilon(\cdot)$ for the inner approximation defined within Lemma~\ref{lemma:approx}.  Specifically, this function depends on both the primal and dual neural networks to quantify an upper bound to the ``distance'' between the inner and outer approximations at any point on the weakly efficient frontier.
\begin{corollary}\label{cor:approx}
The primal neural network $x: \wcal \to \bbx$ provides an $\epsilon(\cdot)$-inner approximation $\cl\co\bigcup_{w \in \wcal} [f(x(w)) + \R^P_+]$ of the upper set $\pcal$ with $\epsilon: \wcal \to \R_+$ defined by
\[\epsilon(w) := w^\T f(x(w)) - d(\lambda(w),w)\]
for any $w \in \wcal$ based, additionally, on the dual neural network $\lambda: \wcal \to \R^M_+$.
\end{corollary}
\begin{proof}
Let $\tilde\pcal := \cl\co\bigcup_{w \in \wcal} [f(x(w)) + \R^P_+]$.  It immediately follows that
\begin{align*}
\cl[\tilde\pcal - \epsilon(w) {\bf 1} + G(w)] &= \left\{y \in \R^P \; | \; w^\T y \geq \inf_{\tilde y \in \tilde\pcal} w^\T [\tilde y - \epsilon(w){\bf 1}]\right\}\\
&= \left\{y \in \R^P \; | \; w^\T y \geq \inf_{\tilde y \in \tilde\pcal} w^\T \tilde y - \epsilon(w)\right\}\\
&\supseteq \left\{y \in \R^P \; | \; w^\T y \geq w^\T f(x(w)) - \epsilon(w)\right\}\\
&= \left\{y \in \R^P \; | \; w^\T y \geq d(\lambda(w),w)\right\}.
\end{align*}
The result trivially follows from Lemma~\ref{lemma:approx}.
\end{proof}

\section{Numerical case studies}\label{sec:cs}
Within this section we will consider four multiobjective problems to numerically study the neural network approach presented above.  The first problem (Section~\ref{sec:2obj}) has 2 objective functions but many decision variables and constraints. In the second problem (Section~\ref{sec:many}), we investigate the performance of the proposed methodology as the number of objectives grow. The third problem (Section~\ref{sec:high}) provides an additional high-dimensional problem so that we can investigate the impact of training time on the performance of the methodology.  The final problem (Section~\ref{sec:mean-var}) presents the mean-variance optimization problem~\cite{markowitz1952} over a large number of assets.  
All computations of the neural networks were completed with PyTorch on a local machine using the Adam optimizer~\cite{kingma2014adam} with learning rate $10^{-4}$. 
Throughout these examples, to generate the loss function, we weight the complimentary slackness condition~\eqref{eq:kkt-cs} with $\eta = 10$.
We wish to note that all hyperparameters -- except the terminal activation functions as presented in Section~\ref{sec:nn} -- are chosen arbitrarily and were not found via, e.g., a grid search. 
All intermediate activation functions (i.e., all but the terminal activation functions) are chosen to be the hyperbolic tangent function with linear linking between layers.

\subsection{Two objective problem}\label{sec:2obj}
Consider the following problem with $P = 2$ objectives:
\begin{align*}
\min\{(\|x\|_2^2 , \|x-{\bf 2}\|_2^2)^\T/N \; | \; x \in [0,1]^N\}.
\end{align*}
Note that this problem was considered as ``Test Instance 2'' in~\cite{eichfelder2021approximation}.
Due to the symmetric structure of this problem, the true unique minimizer for any scalarization $w \in \wcal$ can trivially be deduced as $x^*: \wcal \to [0,1]^N$ defined by
\[x^*(w) := \begin{cases} 2w_2{\bf 1} &\text{if } w_2 \leq \frac{1}{2} \\ {\bf 1} &\text{if } w_2 > \frac{1}{2} \end{cases}.\]
As this problem has a closed form set of minimizers, the efficient frontier can be exactly provided as $\{(\|x^*(w)\|_2^2,\|x^*(w)-2\|_2^2)^\T/N \; | \; w \in \wcal\}$.  This permits us to present an exact comparison between the machine learning methodology to the ground truth.

In order to fully consider our outer approximation as presented in Lemma~\ref{lemma:approx}, we need to also discuss the Lagrange dual problem of the weighted-sum scalarizations.
Let us encode the box constraints through the linear inequalities $Ax \leq b$ for $A := (I , -I)^\T \in \R^{N \times 2N}$ and $b := ({\bf 1}^\T , {\bf 0}^\T)^\T \in \R^{2N}$.  In such a way we consider $M = 2N$ linear inequality constraints.
Due to the quadratic structure of the scalarizations, the Lagrange dual function $d: \R^M \times \wcal \to \R$ can be directly computed by the quadratic structure
\[d(\lambda,w) := -\frac{N}{4}\lambda^\T A A^\T \lambda + (2w_2 A{\bf 1} - b)^\T \lambda + 4w_1w_2.\]

With this setup, we can consider a specific instantiation of the problem.  In particular, for test purposes, we will simply consider this bi-objective problem with $N = 40$ decision variables.  As a direct consequence of the box constraints, there are $M = 80$ inequality constraints.\footnote{The neural network was considered with objective function $x \mapsto N f(x)$ as, numerically, the training process was found to discount the loss associated with the objective too heavily otherwise.}
Furthermore, to demonstrate the power of the neural network methodology proposed above, we train the primal and dual neural networks with only 4 scalarizations: $w \in \{(0,1)^\T,(\frac{1}{3},\frac{2}{3})^\T,(\frac{2}{3},\frac{1}{3})^\T,(1,0)^\T\}$.  Testing will be undertaken on the dense grid $w \in \{(i/1000,1-i/1000) \; | \; i \in \{0,1,...,1000\}\}$.
The primal neural network $x: \wcal \to \bbx$ is designed with three hidden layers, each with 800 hidden nodes; the terminal (projection) activation function is considered with a tolerance of $5 \times 10^{-5}$ to guarantee primal feasibility as discussed in Remark~\ref{rem:tolerance} and with strictly feasible point $\bar x := \frac{1}{2}\times{\bf 1}$.  The dual neural network $\lambda: \wcal \to \bbr^{80}_+$ is designed with three hidden layers, each with 1600 hidden nodes.
All intermediate activation functions (i.e., all but the terminal activation functions) are chosen to be the hyperbolic tangent function with linear linking between layers.
Finally, we train these networks over 1000 epochs. 

\begin{figure}
\centering
\begin{subfigure}[t]{0.45\textwidth}
\centering
\includegraphics[width=\textwidth]{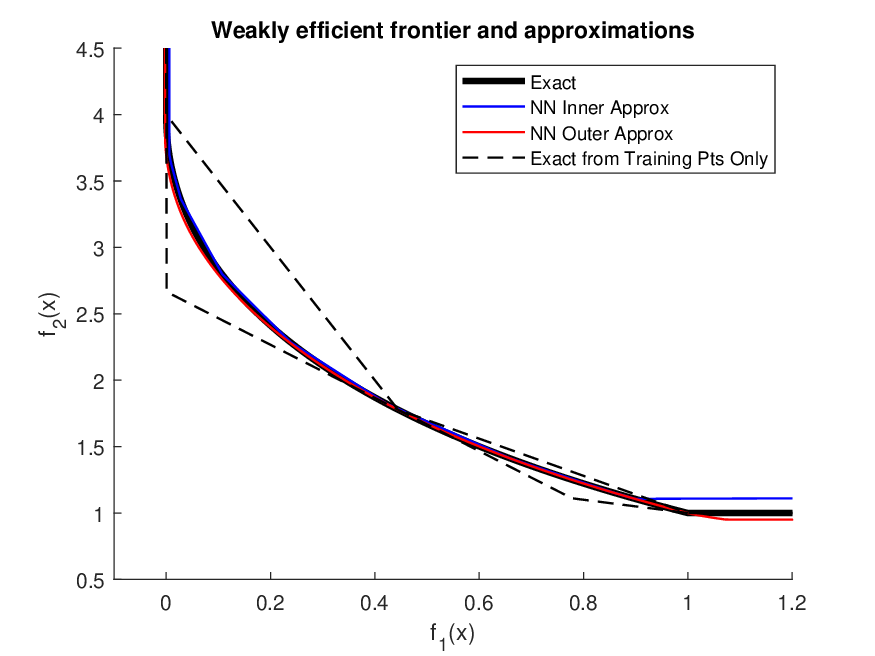}
\caption{Plot of weak efficient frontier (black) and approximations through neural networks (blue and red) or direct computation (black dashed).} 
\label{fig:2obj-efficient}
\end{subfigure}
~
\begin{subfigure}[t]{0.45\textwidth}
\centering
\includegraphics[width=\textwidth]{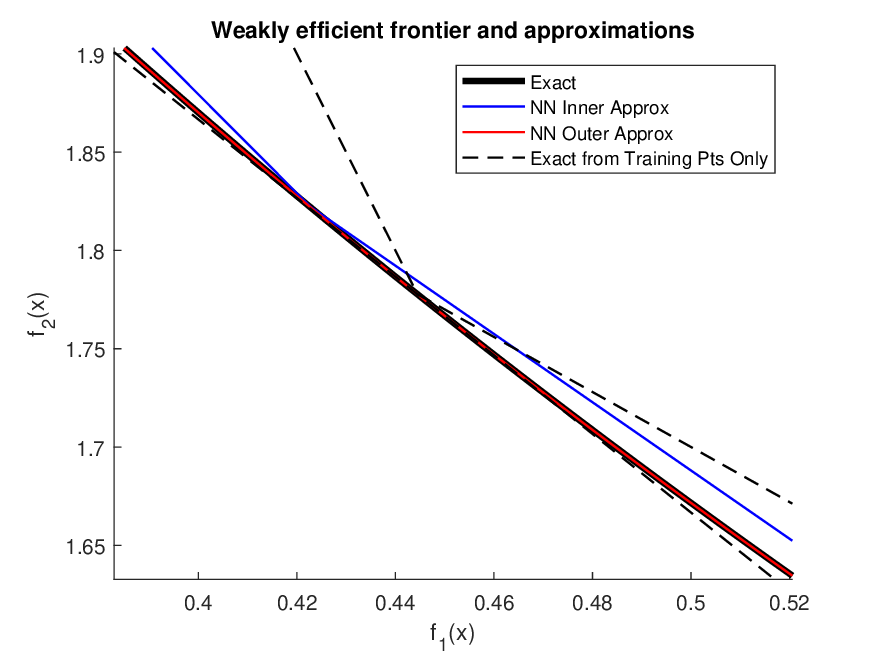}
\caption{Zoomed in plot of the weak efficient frontier (black) and approximations through neural networks (blue and red) or direct computation (black dashed).} 
\label{fig:2obj-efficient-zoom}
\end{subfigure}
~
\begin{subfigure}[t]{0.45\textwidth}
\centering
\includegraphics[width=\textwidth]{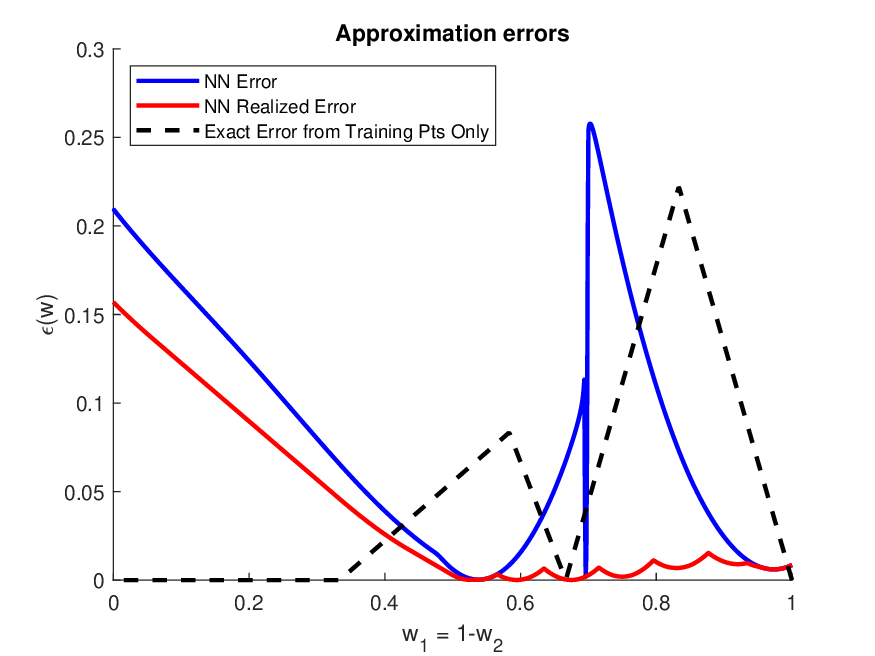}
\caption{Plot of approximation errors $\epsilon(\cdot)$ in linear-scale so that the neural network (blue), realized approximations (red), and direct computation (black dashed) provide $\epsilon(\cdot)$-inner approximations.}
\label{fig:2obj-epsilon-linearscale}
\end{subfigure}
~
\begin{subfigure}[t]{0.45\textwidth}
\centering
\includegraphics[width=\textwidth]{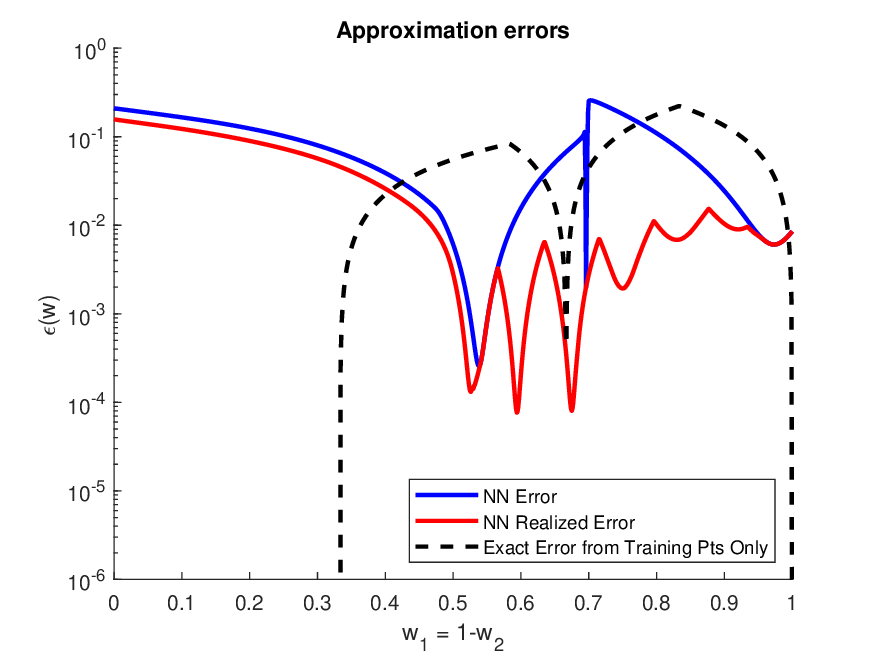}
\caption{Plot of approximation errors $\epsilon(\cdot)$ in log-scale so that the neural network (blue), realized approximations (red), and direct computation (black dashed) provide $\epsilon(\cdot)$-inner approximations.}
\label{fig:2obj-epsilon}
\end{subfigure}
\caption{Section~\ref{sec:2obj}: Plot of the efficient frontier and approximation errors with 1001 regularly spaced test scalarizations. Neural networks and direct computation are computed with the same training scalarizations $w \in \{(0,1)^\T,(\frac{1}{3},\frac{2}{3})^\T,(\frac{2}{3},\frac{1}{3})^\T,(1,0)^\T\}$.}
\label{fig:2obj}
\end{figure}
Figure~\ref{fig:2obj-efficient} displays the true weak efficient frontier along with two approaches for approximation.  The true weak efficient frontier is plotted with a solid black curve.  The inner approximation provided by the primal neural network (i.e., the boundary $\bd\cl\co\bigcup_{w \in \wcal} [f(x(w)) + \R^P_+]$) is plotted as a solid blue line and the outer approximation provided by the dual neural network (i.e., the boundary $\bd\bigcap_{w \in \wcal} \{y \in \R^P \; | \; w^\T y \geq d(\lambda(w),w)\}$) is plotted as a solid red line.  These machine learning approximations are compared with the direct computation of the primal and dual solutions for the 4 training scalarizations as dashed black lines.  Immediately noticeable, the neural network approximations almost completely overlap with each other except in the south east corner of the figure (i.e., for scalarizations $w \in \wcal$ with $w_1$ low).  The approximations can be more clearly seen in Figure~\ref{fig:2obj-efficient-zoom} in which the area around the training scalarization $w = (\frac{2}{3},\frac{1}{3})^\T$ is highlighted.
The errors $\epsilon(w)$ can be directly seen in Figures~\ref{fig:2obj-epsilon-linearscale} and~\ref{fig:2obj-epsilon}; these figures display the exact same data but in a linear and logarithmic scale respectively.  For these scalarizations the error for the exact computation drops to 0, whereas the neural network has errors growing to around 0.2. These errors are further improved (see the red line ``NN Realized Error'' in Figures~\ref{fig:2obj-epsilon-linearscale} and~\ref{fig:2obj-epsilon}) when using the realizations of the neural network approximation -- the convex hull in the inner approximation and the intersection in the outer approximation as provided in Lemma~\ref{lemma:approx}.  However, as is clear from all figures, for (nearly) any choice of scalarization outside ``low'' $w_1$, the neural network \emph{outperforms} the direct computation. Further, we wish to remind the reader that the neural networks were trained without hyperparameter tuning and, as such, the errors $\epsilon(w)$ for the machine learning approach can be further improved.

\subsection{Many objective problem}\label{sec:many}

We now wish to consider how the neural networks can consider a problem with many objectives.  In particular we will consider a problem with $P = M \leq N$ objectives 
\begin{align*}
f_i(x) := (x_i - 1)^2 + \sum_{j \neq i} x_j^2
\end{align*}
for every $i = 1,...,P$ and with $M = P$ constraints $g_j(x) = f_j(x) - 1$ for every $j = 1,...,M$.
Note that this problem was considered as ``Test Instance 4'' in~\cite{eichfelder2021approximation}.\footnote{Within~\cite{eichfelder2021approximation} an additional $2N$ constraints are included so that $x \in [-1000,1000]^N$. However, this bounding box is automatically satisfied by $g_1(x)\leq 0,..., g_M(x) \leq 0$.}
Unlike in the prior example of Section~\ref{sec:2obj}, we do not have a closed form solution to the minimizers to this problem.
However, the neural network methodology presented above can still be applied to deduce inner and outer approximations for the weakly efficient frontier.
To consider the outer approximation, we need to consider the Lagrange dual function $d: \R^M \times \wcal \to \R$ of the weighted-sum scalarizations.  For this problem, the Lagrange dual can be computed as
\[d(\lambda,w) := \sum_{i = 1}^P \left[(w_i + \lambda_i) f_i\left(\frac{w_1+\lambda_1}{\sum_{j = 1}^P [w_j + \lambda_j]} , \cdots , \frac{w_P + \lambda_P}{\sum_{j = 1}^P [w_j + \lambda_j]} , 0 , \cdots , 0 \right) - \lambda_i\right]\]
for any $\lambda \in \R^M$ and $w \in \wcal$.
For test purposes we will vary $P = M \in \{2,...,20\}$ with $N = 100$ within this case study.\footnote{Due to the symmetry of this problem, we implement the primal neural network with $P+1$ outputs providing to the first $P$ variables and then a single value providing the same output $x_i = x_j$ for $i,j \in \{P+1,...,N\}$.}

With this setup, we can consider specific instantiations of the problem at the different choices of $P = M$ (always with fixed $N = 100$).  
Herein, we train each of these problems with 50 uniformly chosen random scalarizations $w \in \wcal$ of the correct dimension.  
Both primal and dual neural networks $x: \wcal \to \bbx$ and $\lambda: \wcal \to \R^M_+$ are designed with 2 hidden layers, each hidden layer with $500$ hidden nodes.  The terminal (projection) activation function for the primal neural network is considered with a tolerance of $5 \times 10^{-5}$ to guarantee primal feasibility as discussed in Remark~\ref{rem:tolerance} and with strictly feasible point $\bar x := \sum_{i = 1}^P \frac{1}{P} \times e_i$ for unit vectors $e_i \in \R^N$.
All intermediate activation functions (i.e., all but the terminal activation functions) are chosen to be the hyperbolic tangent function with linear linking between layers.
Finally, we train these networks over 200 epochs. 

\begin{figure}
\centering
\begin{subfigure}[t]{0.45\textwidth}
\centering
\includegraphics[width=\textwidth]{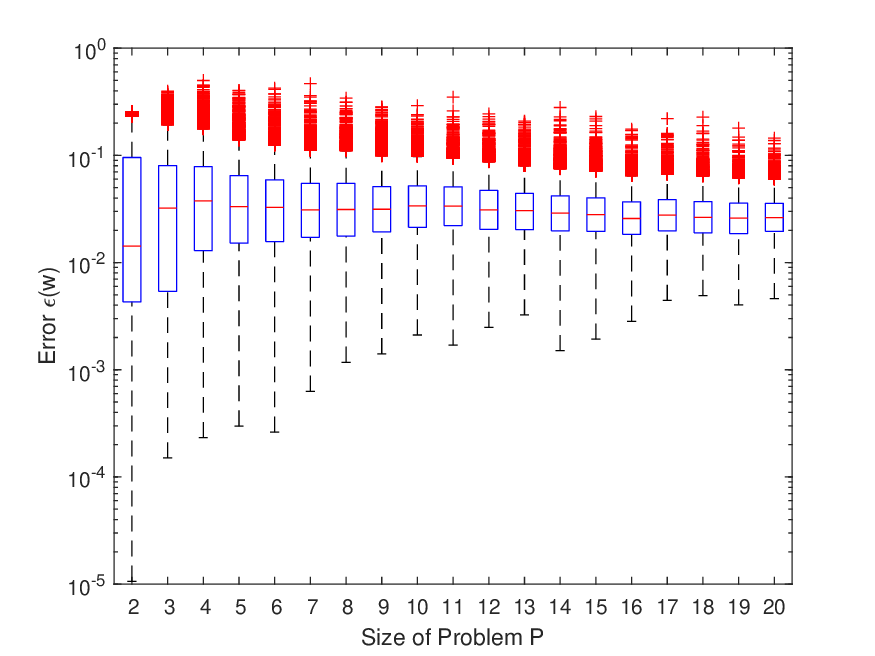}
\caption{Box plot of approximation errors $\epsilon(w)$ from machine learning approach. The central line provides the median; the box shows the inter-quartile range; outliers (red plus-signs) are determined to be further from the median than 1.5 times the inter-quartile range.}
\label{fig:boxplot}
\end{subfigure}
~
\begin{subfigure}[t]{0.45\textwidth}
\centering
\includegraphics[width=\textwidth]{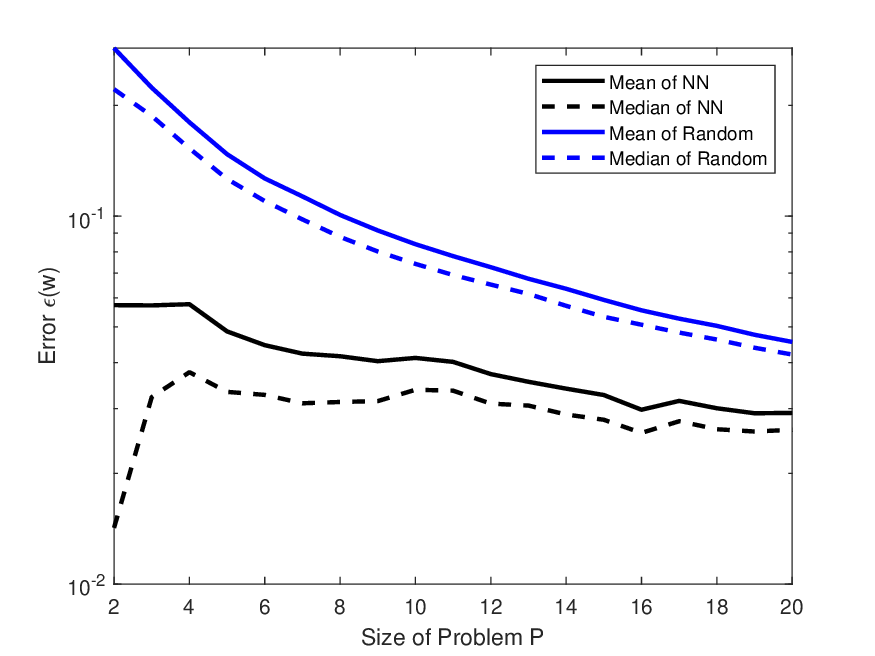}
\caption{Mean (solid) and median (dashed) estimation errors $\epsilon(w)$ from machine learning approach (black) and random primal/dual feasible points (blue).}
\label{fig:NNvsRand}
\end{subfigure}
~
\begin{subfigure}[t]{0.45\textwidth}
\centering
\includegraphics[width=\textwidth]{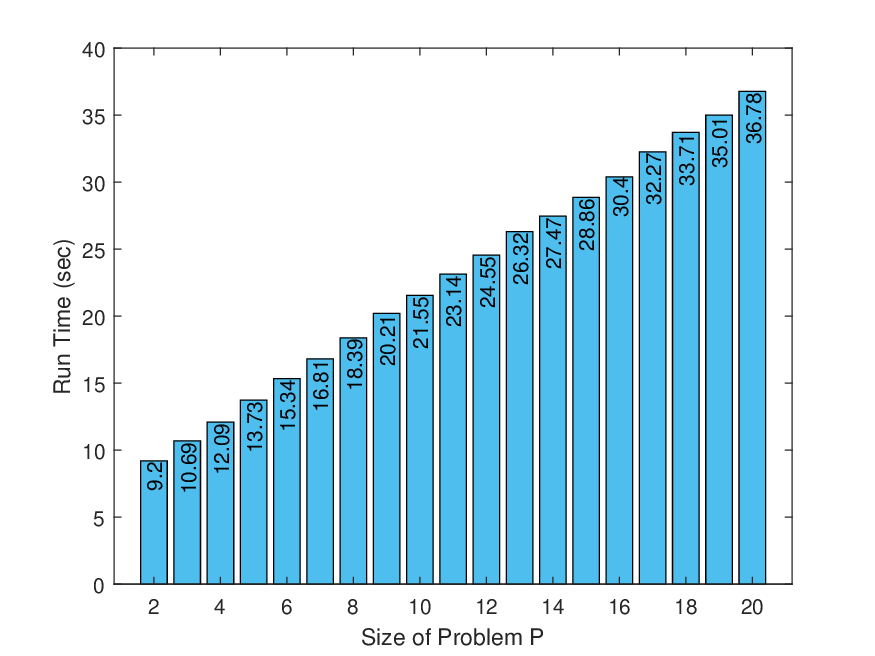}
\caption{Runtime for training the neural networks and computing the test scalarizations.}
\label{fig:time}
\end{subfigure}
\caption{Section~\ref{sec:many}: Visualizations of approximation errors $\epsilon(w)$ over 5000 (uniformly random) test scalarizations $w \in \wcal$ as the size of the problem $P = M$ changes with fixed number of primal variables $N = 100$ and associated runtimes when trained over 50 (uniformly random) scalarizations.}
\label{fig:many}
\end{figure}
In order to compare the performance of the inner and outer approximations as the size of the problem grows, we compute $\epsilon(w)$ for 5000 uniformly sampled scalarizations $w \in \wcal$ of the correct dimension.  These errors (in logarithmic scale) are displayed in Figure~\ref{fig:boxplot}.  Note that the errors $\epsilon(w)$ become more consistent as the dimension of the problem increases; this is noticeable with the shrinking of the error bars of the box plot.  
Additionally, as can be seen in Figure~\ref{fig:NNvsRand}, both the mean and median of these errors (tend to) decrease as the dimension of the problem grows.    
This improvement in the mean and median of the errors is primarily due to the shrinking size of the feasible region as the dimension grows.  For instance, as the problem size $P = M$ grows, the feasible region can be proven to exist within a box with shrinking volume, i.e., $\bbx \subseteq [0,1]^P \times [-1,1]^{N-P}$.  This is made clear in Figure~\ref{fig:NNvsRand} as the inner approximation from selecting a random primal feasible point (uniformly selected in $[0,1]^P \times \{0\}^{N-P}$ and projected onto $\bbx$ via the projection activation function proposed in Proposition~\ref{prop:primal}) and random dual feasible point (deterministically selecting $\lambda := {\bf 0}$) improves as the size of the problem increases.  
Importantly, the neural networks uniformly outperform the random approximations for all tested problems.
Finally, we wish to highlight that the reported approximation errors $\epsilon(w)$ in Figure~\ref{fig:many} are without considerations of the convex hull for the inner approximation and intersection for the outer approximation as in Lemma~\ref{lemma:approx}.  For this reason, we view the reported errors as an upper bound on those that would be realized by the neural networks constructed in this example. 

Beyond demonstrating that the proposed neural network approach is able to approximate the weak efficient frontier, we also use this many objective problem to comment on the computational requirements for training the primal and dual neural networks.  The computational runtimes (on a local machine) for training and testing the neural networks are displayed in Figure~\ref{fig:time}.  Though this methodology does take longer as the problem size $P = M$ grows, it does so in a relatively predictable manner.  The increased time is due to the growth in the size of both the primal and dual neural networks.  
Though a one-to-one comparison is not possible as the times reported in~\cite{eichfelder2021approximation} were run on a different machine than our neural network approach and were computed with fixed approximation errors, we conclude this discussion by directly comparing computation times and approximation errors with the reported results in~\cite{eichfelder2021approximation}. 
Specifically, we find that the neural network approach has longer runtime for $P = M \in \{2,3\}$ (neural network training required 9.2 and 10.69 seconds respectively) than the box coverage approach of~\cite{eichfelder2021approximation} (2.46 and 4.64) for fixed approximation error $\epsilon = 0.2$; with that same approximation error, already at $P = M = \{4,5\}$, the neural network approach outperforms the box coverage approach of~\cite{eichfelder2021approximation} both in runtime (12.09 and 13.73 seconds against 15.05 and 122.25 seconds respectively) with improved accuracy as well as $\mathbb{P}(\epsilon(w) < 0.2) \approx 98\%$ for uniformly sampled test scalarizations.
Improving the approximation error in the box coverage approach of~\cite{eichfelder2021approximation} to fixed error $\epsilon = 0.1$ (which notably is still worse than the proposed neural network approach on average in all test problems), the neural network approach is faster for $P = M =3$ (runtime of 22.90 seconds compared to 12.09 seconds for the neural network approach). In fact, demonstrating the curse of dimensionality in the box coverage approach, that method could not complete its computations within an hour even for $P = M = 4$ (compared to a runtime of 13.73 seconds for our neural network approach).  Even for $P = M = 20$, the neural network approach still completes all computations in under 37 seconds on a local machine.\footnote{Higher dimensional cases were not provided as the feasible region becomes too small.}

\subsection{High dimensional problem}\label{sec:high}
We now wish to consider how the neural networks can consider a problem with an arbitrarily large number of objectives and variables.  In particular we will consider a problem with $P = N$ objectives 
\begin{align*}
f_i(x) := x_i^2
\end{align*}
for every $i = 1,...,P$ and with $M = 1$ constraints $g_1(x) = \|x - (1+\epsilon){\bf 1}\|_2 - 1$ for fixed $\epsilon > 0$.  Throughout this section, we take $\epsilon = 0.01$ chosen arbitrarily.
Similar to the many objective problem in Section~\ref{sec:many}, we do not have a closed form solution to the minimizers to this problem.
However, the neural network methodology presented above can still be applied to deduce inner and outer approximations for the weakly efficient frontier.
To consider the outer approximation, we need to consider the Lagrange dual function $d: \R^M \times \wcal \to \R$ of the weighted-sum scalarizations.  For this problem, the Lagrange dual can be computed as
\[d(\lambda,w) := \inf_{x \in \R^N} \left[\sum_{i = 1}^N w_i x_i^2 + \lambda(\|x - (1+\epsilon){\bf 1}\|_2 - 1)\right]\]
for any $\lambda \in \R^M$ and $w \in \wcal$.
This problem was previously formulated as ``Problem 7.2'' in~\cite{bokrantz2013algorithm}.\footnote{Within \cite{bokrantz2013algorithm} the constraint is presented with the squared norm instead.}
For test purposes, in both case studies, we will vary $P = N \in \{2,...,10,15,...,50,60,...,100,500,1000,5000\}$.

We train all of these problems with 50 uniformly chosen random scalarizations $w \in \wcal$ of the correct dimension.  
Both primal and dual neural networks $x: \wcal \to \bbx$ and $\lambda \to \R^M_+$ are designed with 2 hidden layers, each hidden layer with $300$ hidden nodes.  The terminal (projection) activation function for the primal neural network is considered with a tolerance of $5 \times 10^{-5}$ to guarantee primal feasibility as discussed in Remark~\ref{rem:tolerance} and with strictly feasible point $\bar x := (1+\epsilon){\bf 1}$ for the one vector ${\bf 1} \in \R^N$.\footnote{For the implementation we follow the modified problem as discussed in Remark~\ref{rem:center0}.} 
In contrast to the prior numerical examples, here we use the soft plus activation function (i.e., $\log(1 + \exp(\cdot))$) to guarantee dual feasibility.
All intermediate activation functions (i.e., all but the terminal activation functions) are chosen to be the hyperbolic tangent function with linear linking between layers.
In order to compare the performance of the inner and outer approximations as the size of the problem grows, we compute $\epsilon(w)$ for 5000 uniformly sampled scalarizations $w \in \wcal$ of the correct dimension over varying number of epochs for the Adam optimizer.

\begin{figure}
\centering
\begin{subfigure}[t]{0.45\textwidth}
\centering
\includegraphics[width=\textwidth]{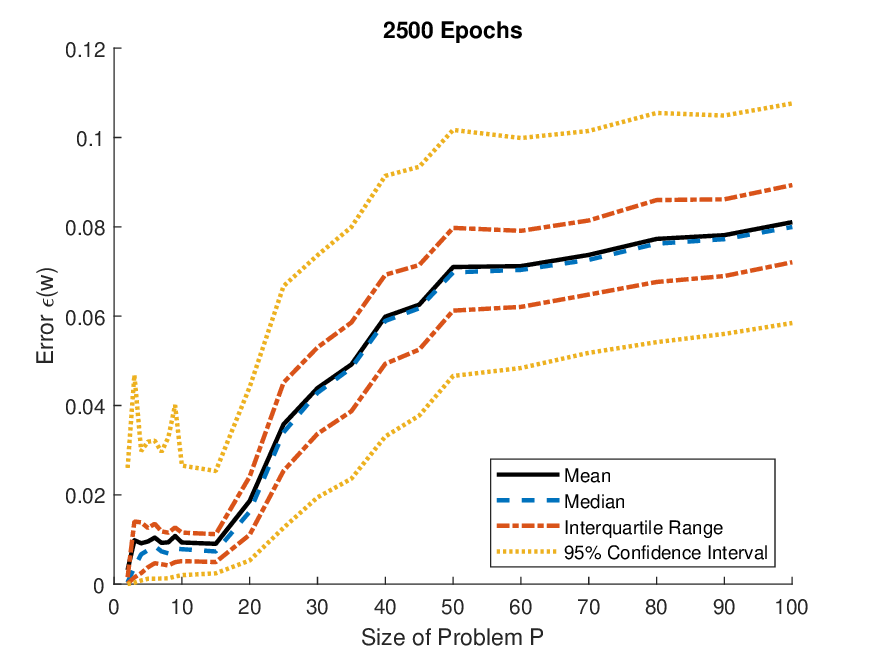}
\caption{Mean (solid), median (dashed), and confidence intervals (dash-dot: 50\% and dotted: 95\%) of the test approximation errors $\epsilon(w)$ from the machine learning approach when trained over 2500 epochs.}
\label{fig:NNerror72}
\end{subfigure}
~
\begin{subfigure}[t]{0.45\textwidth}
\centering
\includegraphics[width=\textwidth]{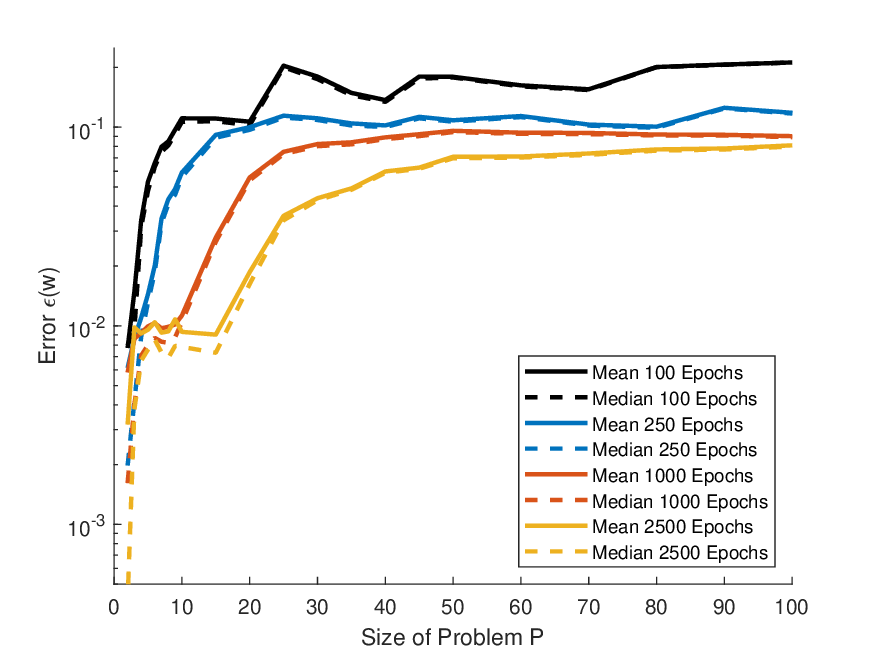}
\caption{Comparison of mean (solid) and median (dashed) of the test approximation errors $\epsilon(w)$ from the machine learning approach when trained over varying number of epochs.}
\label{fig:NNerror72epochs}
\end{subfigure}
~
\begin{subfigure}[t]{0.45\textwidth}
\centering
\includegraphics[width=\textwidth]{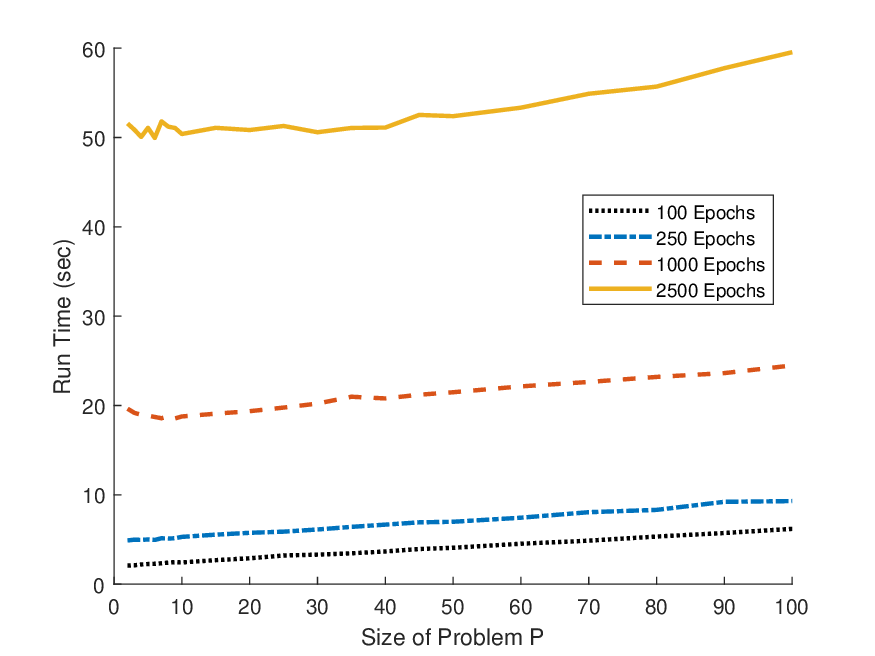}
\caption{Runtime for training the neural networks and computing the test scalarizations for varying number of epochs.}
\label{fig:time72}
\end{subfigure}
\caption{Section~\ref{sec:high}: Visualizations of approximation errors $\epsilon(w)$ over 5000 (uniformly random) test scalarizations $w \in \wcal$ as the size of the problem $P = N \leq 100$ and number of epochs of the Adam optimizer change as well as the associated runtimes when trained over 50 (uniformly random) scalarizations.}
\label{fig:high}
\end{figure}

In Figure~\ref{fig:NNerror72}, the test approximation errors are plotted as the size of the problem $P = N \leq 100$ varies after training for 2500 epochs of the Adam optimizer. Note that the errors are (approximately) an order of magnitude lower for $P = N = 2$ than $P = N = 3$ and beyond, i.e., on the order of $10^{-3}$ growing to $10^{-2}$.  These errors stabilize around $10^{-2}$ until $P = N = 15$ after which the errors grow (but level off) as the dimensions increase. This behavior can be seen in the mean, median, the interquartile range, and the 95\% confidence interval.
We can benchmark these errors in two ways:
\begin{enumerate}
\item If, in comparison, we naively chose the feasible points $\hat{x}(w) := (1+\epsilon){\bf 1}$ and $\hat{\lambda}(w) := 0$, then the errors would uniformly be the constant $\hat{\epsilon}(w) := (1+\epsilon)^2 > 1$. As such these neural networks vastly outperform the naive choice for all tested problems. Notably, randomly selecting primal feasible points uniformly results in comparable errors.
\item Alternatively, if we chose the $\hat{x}(w) := (1+\epsilon - 1/\sqrt{N}){\bf 1}$ (so that we select a point on the efficient frontier) and $\hat{\lambda}(w) := 0$, then the errors would depend on the dimension of the problem as $\hat{\epsilon}_N(w) := (1+\epsilon - 1/\sqrt{N})^2$. For $N = 2$ with our chosen $\epsilon = 0.01$, this results in an error of approximately $0.0917$ which is nearly 2 orders of magnitude larger than the test errors found by our neural network approach; as $N$ grows so do these potential errors always remaining significantly above the 95\% confidence interval as displayed in Figure~\ref{fig:NNerror72}.
\end{enumerate}
As with the prior example of Section~\ref{sec:many}, the neural networks could be further improved as the reported approximation errors $\epsilon(w)$ in Figure~\ref{fig:NNerror72} are without considerations of the convex hull for the inner approximation and the intersection for the outer approximation as in Lemma~\ref{lemma:approx}. For this reason, we view the reported errors as an upper bound on those that would be realized by the neural networks constructed in this example.

In Figure~\ref{fig:NNerror72epochs}, we explore the impact of training the neural networks on the test performance. In particular, we vary the number of epochs over which we train both the primal and dual neural networks between 100 and 2500. Two key results are observable. First, as we increase the amount of epochs over which we run the Adam optimizer, the test errors decrease. Note that these errors are displayed in a log-scale so that the errors decrease a significant amount as we increase the training time. Second, errors of around $10^{-2}$ can be obtained by training over more epochs; this error level can be obtained for larger dimensions if we increased the training time further. Notably this improved test performance in higher dimensions comes only from lengthening the number of training epochs while retaining the constant 50 (random) training scalarizations.

Beyond demonstrating that the proposed neural network approach is able to approximate the weak efficient frontier, we also use this problem to comment on the computational requirements for training the primal and dual neural networks.  The computational runtimes (on a local machine) for training and testing the neural networks (with $P = N \leq 100$) are displayed in Figure~\ref{fig:time72}.  
Though this methodology does take longer as the problem size $P = N$ grows, it does so at a nearly linear growth rate. If we expand beyond $P = N \geq 100$, we find that the linear rate remains up to the 500-dimensional problem (47.86 seconds); however the runtime grows superlinearly after around $P = N = 500$. For $P = N = 1000$, training the primal and dual neural networks took approximately 3 minutes and 43 seconds for 1000 epochs; for $P = N = 5000$, training these neural networks required 39 minutes and 16 seconds for 1000 epochs.
The increased time is due to the growth in the size of the primal dual neural network.  

As in Section~\ref{sec:many} above, a one-to-one comparison is not possible as the times reported in~\cite{bokrantz2013algorithm} were run on a different machine than our neural network approach, we conclude this discussion by directly comparing computation times and approximation errors with the reported results in~\cite{bokrantz2013algorithm}. 
Specifically, we find that the neural network approach has similar runtimes for $P = N \in \{2,3\}$ (approximately 10 seconds following \cite{bokrantz2013algorithm}). However, whereas the neural network methodology has nearly linear growth in runtime (for $P = N \leq 500$), the runtime for the methodology proposed within \cite{bokrantz2013algorithm} grows exponentially in the problem dimensions. In fact, already at only 14 dimensions, the methodology proposed within \cite{bokrantz2013algorithm} required over an hour in order to reach an error of approximately $0.25$ (i.e., a longer time than required for the neural networks in the 5000-dimensional setting).

\subsection{Mean-variance optimization}\label{sec:mean-var}
Consider now the mean-variance portfolio optimization problem proposed originally by Markowitz~\cite{markowitz1952} as a bi-objective optimization problem.  Specifically, this problem consists of a portfolio optimizer who is simultaneously attempting to maximize her expected return and minimize her risk (as encoded by the variance of the portfolio returns).  We will present this problem with no-short selling allowed, i.e., the investor can only purchase assets for her portfolio.  As such we will consider the following $P = 2$ objectives
\[f_1(\hat{x}) = -\bar{r}^\T \hat{x} \quad \text{ and } \quad f_2(\hat{x}) = \frac{1}{2}\hat{x}^\T C \hat{x}\]
over $S \geq 2$ assets in which the agent can invest, and where $\bar{r} \in \R^S$ denotes
the mean returns for each asset  and $C \in \R^{S \times S}$ is the covariance structure.  The investor is constrained by $M = S$ inequality constraints $g_j(\hat{x}) = -\hat{x}_j$ for every $j = 1,...,M$ (encoding the no-short selling constraints).
In contrast to the prior examples, herein the optimizer is constrained so that she uses all funds that she has available (i.e., ${\bf 1}^\T \hat{x} = 1$).
As discussed within Remark~\ref{rem:convex}\eqref{rem:convex-equality}, we consider the modification of this problem to guarantee the equality constraint is imposed throughout; specifically, we consider $N = S-1$ variables with $\hat{x} := (x^\T , 1-{\bf 1}^\T x)^\T$ for $x \in \R^N$.
To guarantee that this problem is strictly convex we will restrict the set of scalarizations under consideration to $\wcal_+ := \{w \in \wcal \; | \; w_2 > 0\}$ with positive definite covariance matrix $C$ as is discussed in Remark~\ref{rem:convex}\eqref{rem:convex-convex}.

To consider the outer approximation, we need to consider the Lagrange dual function $d: \R^M \times \wcal_+ \to \R$ of the weighted-sum scalarizations.  For this mean-variance problem, the Lagrange dual can be computed as
\begin{align*}
d(\lambda,w) &:= w^\T f(A_{\perp} \tilde{x}^*(\lambda,w) + e_s) - \lambda^\T(A_{\perp}\tilde{x}^*(\lambda,w) + e_s),\\
\tilde{x}^*(\lambda,w) &:= (w_2 A_{\perp}^\T C A_{\perp})^{-1} \left[A_{\perp}^\T \lambda + w_1(\bar{r}_{-S} - \bar{r}_S{\bf 1}) - w_2 A_{\perp}^\T C e_s\right],\\
A_{\perp} &:= \left(I_{N \times N} \, , \, -{\bf 1}\right)^\T \in \R^{S \times N}, \qquad e_s := ({\bf 0}^\T , 1)^\T \in \R^S
\end{align*}
for any $\lambda \in \R^M$ and $w \in \wcal_+$.

We will implement this mean-variance optimization problem on $S = 492$ stocks in the S\&P500 index.\footnote{These 492 stocks are currently in the S\&P500 index and had daily data available on Yahoo Finance from January 1, 2018 through December 31, 2021.} Empirical means and covariances for daily log returns from January 1, 2018 through December 31, 2021 are used to construct the mean asset returns $\bar{r}$ and covariance matrix $C$.\footnote{The neural network was considered with adjusted empirical mean returns and covariance matrix so that the objective values have comparable orders of magnitude. Specifically, the mean returns are multiplied by $100^2/6$ and the covariance matrix is multiplied by $100^2$. Without this modification of $1/6$, the training process was found to overly weight the importance of the mean returns and consistently results in the maximum return portfolio. Furthermore, the multiplication of both objectives by $100^2$ was used to eliminate numerical issues with the invertibility of the covariance matrix $C$ for computing the outer approximation.}
Herein, we train this problem with $5$ scalarizations: $w \in \{(0,1)^\T, (\frac{1}{4},\frac{3}{4})^\T, (\frac{1}{2},\frac{1}{2})^\T, (\frac{3}{4},\frac{1}{4})^\T, (1-10^{-5},10^{-5})^\T\}$. We then test this problem over 1001 scalarizations $w \in \{(i/1000,1-i/1000) \; | \; i \in \{0,1,...,1000\}\}$.  The primal neural network $x: \wcal_+ \to \bbx$ is designed with three hidden layers, each with 800 hidden nodes.  Due to the construction of the feasible region, we apply a ReLU activation function prior to the projection activation function, i.e., such that the neural network $z$ in Proposition~\ref{prop:primal} has a terminal ReLU activation function.  The terminal (projection) activation function for this primal neural network is considered with tolerance of $5 \times 10^{-5}$ to guarantee primal feasibility as discussed in Remark~\ref{rem:tolerance} and with strictly feasible point $\bar x := 7.5\times 10^{-5} {\bf 1}$.  The dual neural network $\lambda: \wcal_+ \to \R^{492}_+$ is designed with three hidden layers, each with 800 hidden nodes.  As with the high dimensional example in Section~\ref{sec:high}, here we use the soft plus activation function (i.e., $\log(1 + \exp(\cdot))$) to guarantee dual feasibility.  All intermediate activation functions (i.e., all but the terminal activation functions) are chosen to be the hyperbolic tangent function with linear linking between layers.  Finally, we train these networks over 5000 epochs.

\begin{figure}
\centering
\begin{subfigure}[t]{0.45\textwidth}
\centering
\includegraphics[width=\textwidth]{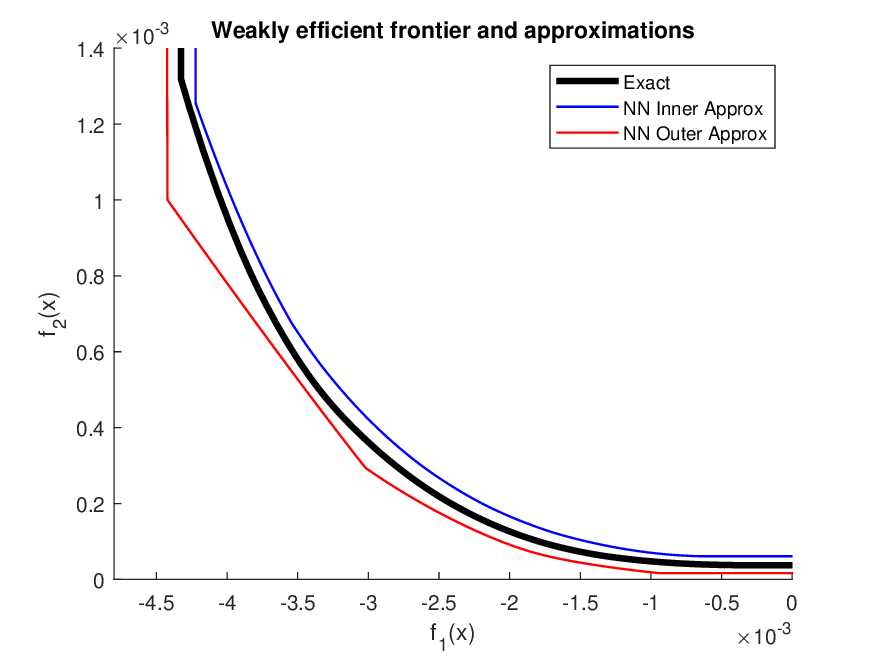}
\caption{Plot of weak efficient frontier (black) and approximations through neural networks (blue and red).} 
\label{fig:mv-efficient}
\end{subfigure}
~
\begin{subfigure}[t]{0.45\textwidth}
\centering
\includegraphics[width=\textwidth]{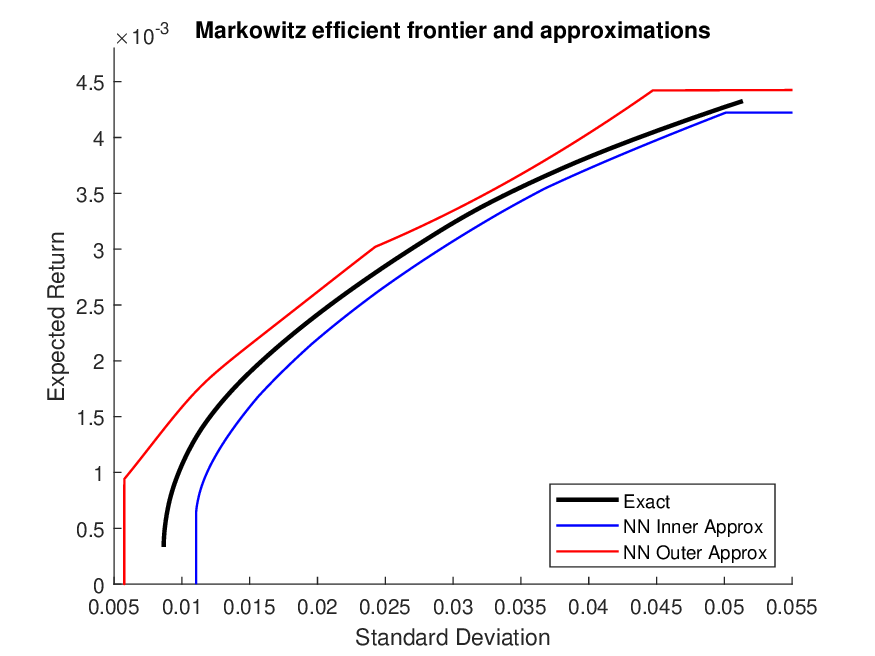}
\caption{Plot of the Markowitz bullet (black) and approximations through neural networks (blue and red).}
\label{fig:mv-bullet}
\end{subfigure}
~
\begin{subfigure}[t]{0.45\textwidth}
\centering
\includegraphics[width=\textwidth]{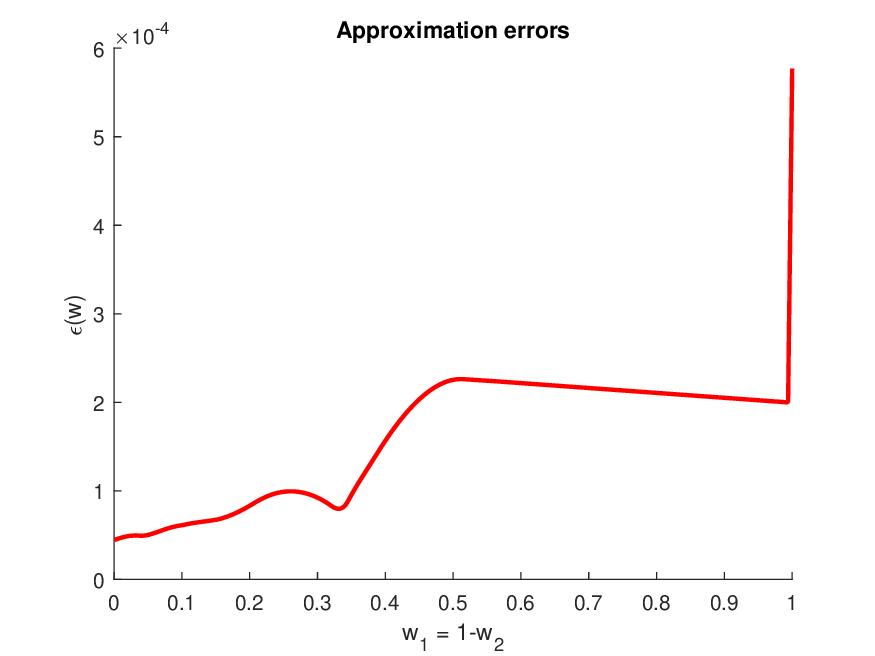}
\caption{Plot of the realized approximation errors $\epsilon(\cdot)$.}
\label{fig:mv-epsilon}
\end{subfigure}
\caption{Section~\ref{sec:mean-var}: Plot of the efficient frontier, Markowitz bullet, and approximation errors over 1001 regularly spaced test scalarizations. Neural networks are computed with 5 regularly spaced training scalarizations.}
\label{fig:mv}
\end{figure}

Figure~\ref{fig:mv-efficient} displays the true weak efficient frontier along with the neural network inner and outer approximations.  The true weak efficient frontier is plotted with a solid black curve.  The inner approximation provided by the primal neural network (i.e., the boundary $\bd\cl\co\bigcup_{w \in \wcal} [f(x(w)) + \R^2_+]$) is plotted as a solid blue line and the outer approximation provided by the dual neural network (i.e., the boundary $\bd\bigcap_{w \in \wcal_+} \{y \in \R^2 \; | \; w^\T y \geq d(\lambda(w),w)\}$) is plotted as a solid red line.  
Using the usual transformation of the the mean-variance problem, the Markowitz bullet approximations providing the relation between the standard deviation and mean of the portfolio return, are displayed in Figure~\ref{fig:mv-bullet}.
As shown in Figure~\ref{fig:mv-epsilon}, the realized approximation errors is on the order of $10^{-4}$ throughout the space of scalarizations $w \in \wcal_+$.  Note that the approximation errors jump up for $w_2 \approx 0$, i.e., when the problem approaches a linear objective (rather than strictly convex as imposed in Assumption~\ref{ass:slater}).

\section{Extending to non-strictly convex vector optimization problems}\label{sec:convex}
As stated within Remark~\ref{rem:convex}\eqref{rem:convex-convex}, within this section we wish to relax the strict convexity assumption of Assumption~\ref{ass:slater} so that we only require $f_i$ to be convex for every objective function $i = 1,...,P$.
That is, consider the CVOP~\eqref{eq:cvop} such that 
\begin{itemize}
\item $f_i$ are convex and continuously differentiable;
\item $g_j$ are convex and continuously differentiable;
\item \eqref{eq:cvop} is bounded;
\item Slater's condition holds with strictly feasible point $\bar x \in \R^N$.
\end{itemize}
With this set of assumptions, which differ from Assumption~\ref{ass:slater} only through the non-strict convexity of the objective functions, we again consider the weighted-sum scalarization problem~\eqref{eq:scalarization} and its dual~\eqref{eq:scalarization-dual} in order to learn the weakly efficient frontier. 

The impact of the potential non-strict convexity of the objective functions to the machine learning approach can be two-fold.
\begin{itemize}
\item
In contrast to the simple dual feasibility conditions~\eqref{eq:kkt-dual} under Assumption~\ref{ass:slater}, we might need to consider additional implicit dual feasibility conditions to guarantee finiteness of the Lagrange dual function~\eqref{eq:dual}. As this depends on the particular structure of the considered problem, we will discuss this in more detail in the applications where this issue occurs. The case of linear vector optimization problems are, for instance, discussed in Section~\ref{subsec:d}.
\item
The approach considered in Section~\ref{sec:nn} needs to be modified as neural networks can be unstable for discontinuous functions (e.g., the universal approximation theorem (Theorem~\ref{thm:uat}) only applies for continuous functions).  Specifically, due to the possibility of multiplicity of optimizers for convex problems, there may not exist a continuous selector of optimizers.
\end{itemize}
To address the second point above, the approach taken herein consists of constructing feasible primal and dual neural networks to the original problem~\eqref{eq:cvop} but where the risk functional \eqref{eq:loss} is based on the KKT conditions of the scalarizations of an \emph{augmented} CVOP with scalarization
\begin{equation}\label{eq:scalarization-convex}\tag{$\overline{\text{wP}}$}
\min\Big\{(1-\delta) \sum_{i = 1}^P w_i f_i(x) + \delta \bar f(x) \; | \; x \in \bbx\Big\},
\end{equation}
where $w \in \wcal$ with $\delta \in (0,1)$ small and $\bar f: \R^N \to \R$ \emph{strictly convex} and continuously differentiable.  In particular, setting $\bar f(x) := \|x\|_{2}^2$ implies that the objective of~\eqref{eq:scalarization-convex} is strongly convex and thus also strictly convex.

As demonstrated by Proposition~\ref{prop:berge} below, this augmentation provides a reliable approximation to the solution of the original weighted-sum scalarization~\eqref{eq:scalarization}.
\begin{proposition}\label{prop:berge}
Assume $\bbx \subseteq \R^N$ is compact and fix $w \in \wcal$.
Let $X^*(w) \subseteq \bbx$ be the set of primal solutions of the scalarization problem~\eqref{eq:scalarization} w.r.t.\ $w$.
Additionally, let $\bar{x}^*(w,\delta)$ be the primal solutions of the scalarization problem~\eqref{eq:scalarization-convex} w.r.t.\ $w$ and $\delta \in (0,1)$.   
Then there exists at least one accumulation point of $\{\bar{x}^*(w,\delta) \; | \; \delta \in (0,1)\}$ as $\delta \searrow 0$ and any such accumulation point is an element of $X^*(w)$, i.e., $\lim_{\delta \searrow 0} \bar{x}^*(w,\delta) \subseteq X^*(w)$.
\end{proposition}
\begin{proof}
Consider the extension of~\eqref{eq:scalarization-convex} to include $\delta \in [0,1)$, i.e., such that \eqref{eq:scalarization-convex} reduces to the original weighted-sum scalarization \eqref{eq:scalarization} at $\delta = 0$.  Let $\bar{X}^*(w,\delta) \subseteq \bbx$ denote the set of optimizers for~\eqref{eq:scalarization-convex} for any $\delta \in [0,1)$, i.e.,
\[\bar{X}^*(w,\delta) := \begin{cases} X^*(w) &\text{if } \delta = 0 \\ \{\bar{x}^*(w,\delta)\} &\text{if } \delta \in (0,1). \end{cases}\]
By the Berge maximum theorem (see, e.g.,~\cite[Theorem 17.31]{AB07}), $\delta \in [0,1) \mapsto \bar{X}^*(w,\delta)$ is upper continuous.  In particular, by the closed graph theorem (see, e.g.,~\cite[Theorem 17.11]{AB07}), the graph of $\delta \in [0,1) \mapsto \bar{X}^*(w,\delta)$ is closed which implies any accumulation point $\bar{x}^*(w)$ of $\{\bar{x}^*(w,\delta) \; | \; \delta \in (0,1)\}$ as $\delta \searrow 0$ is an element of $X^*(w)$.   
Finally, since $\{\bar{x}^*(w,\delta) \; | \; \delta \in (0,1)\} \subseteq \bbx$, there must exist at least one accumulation point as $\delta \searrow 0$.
\end{proof}

\begin{remark}
Under suitable additional assumptions for the uniqueness of the dual solution $\bar{\lambda}^*(w,\delta)$ to the Lagrange dual problem of~\eqref{eq:scalarization-convex}, a comparable statement to Proposition~\ref{prop:berge} can be given for $\bar{\lambda}^*(w,\delta) \to \lambda^*(w)$. 
\end{remark}

To demonstrate this approach, we will first illustrate the exact construction for a generic linear vector optimization problems (LVOP) in Section~\ref{sec:convex-lvop} along with a simple numerical example in that setting.  We will then consider in Section~\ref{sec:mr} a version of the mean-risk problem to present a separate, practical problem from financial applications.

\subsection{Linear vector optimization problems}\label{sec:convex-lvop}
Consider the generic inequality-constrained LVOP
\[\min\{Cx \; | \; Ax \leq b\}\]
for $C \in \R^{P \times N},~A \in \R^{M \times N},~b \in \R^M$.  That is, $\bbx := \{x \in \R^N \; | \; Ax \leq b\}$.\footnote{Equality constraints can be included as discussed within Remark~\ref{rem:convex}\eqref{rem:convex-equality}.}
By construction, this is a CVOP that is \emph{not} strictly convex.  We will use this problem to illustrate the modifications that are necessary to formalize the neural network approach provided within Section~\ref{sec:nn} to guarantee the inner and outer approximation results provided within Section~\ref{sec:approx}.
Briefly, the idea is to construct primal and dual feasible neural networks to the LVOP that approximate the solution to a closely related \emph{strictly convex} CVOP.

Using this LVOP, let us construct our primal and dual neural networks 
followed by a discussion of the necessary changes to the loss function for training purposes.  We wish to note that the inner and outer approximation results provided within Section~\ref{sec:approx} can be applied to the constructions provided within this section without modification.

\subsubsection{Primal feasible neural network}\label{subsec:p}
Consider the weighted-sum scalarization problem
\[\min\{w^\T Cx \; | \; Ax \leq b\}\]
for $w \in \wcal$.
Herein we construct the primal feasible neural network $x: \wcal \to \bbx$ exactly as in Section~\ref{sec:nn-primal} and, specifically, as detailed in Proposition~\ref{prop:primal}.  
That is, let $z: \wcal \to \R^N$ define an arbitrary neural network.  Let $x: \wcal \to \R^N$ such that
\[x(w) := (1-t^*(z(w))) z(w) + t^*(z(w)) \bar x\]
where $t^*: \R^N \to [0,1]$ is defined as in~\eqref{eq:t}.

\subsubsection{Dual feasible neural network}\label{subsec:d}
Consider the Lagrange dual problem to the weighted-sum scalarization problem w.r.t.\ $w \in \wcal$.  Herein we need to be careful with the domain of the Lagrange dual function $d: \R^M \times \wcal \to \R$; specifically, we can explicitly compute $d$ as
\begin{align*}
d(\lambda,w) := \inf_{x \in \R^N} \left[w^\T C x + \lambda^\T(Ax - b)\right] = \begin{cases} -b^\T \lambda &\text{if } A^\T\lambda = -C^\T w \\ -\infty &\text{else}\end{cases}
\end{align*}
for any $\lambda \in \R^M,~w \in \wcal$.  Therefore the dual problem includes the additional equality constraint $A^\T\lambda = -C^\T w$, i.e.,~\eqref{eq:scalarization-dual} can be specified as
\[\max\{-b^\T\lambda \; | \; A^\T\lambda = -C^\T w,~\lambda \in \R^M_+\}.\]
For notational simplicity, define $\Lambda(w) := \{\lambda \in \R^M \; | \; -\lambda \in -\R^M_+,~A^\T \lambda = -C^\T w\}$ to provide the set of feasible dual solutions $\lambda$. 

\begin{assumption}\label{ass:dual-feasible}
There exists a strictly feasible point $\bar{\lambda}(w) \in \Lambda(w)$ for every $w \in \wcal$.  That is, $A^\T\bar{\lambda}(w) = -C^\T w$ with $\bar{\lambda}(w) \in \R^M_{++}$ for every $w \in \wcal$. 
\end{assumption}

\begin{remark}\label{rem:dual_strict_feasible}
A strictly feasible dual point $\bar{\lambda}(w) \in \Lambda(w)$ can be found by solving the linear program
\[\bar{\lambda}(w) \in \argmin\{\sum_{i = 1}^M \lambda_i \; | \; -\lambda + \bar{\epsilon} \in -\R^M_+,~A^\T \lambda = -C^\T w\}\]
for some $\bar{\epsilon} \in \R^M_{++}$.  Due to the existence of a strictly feasible point (Assumption~\ref{ass:dual-feasible}), so long as $\bar{\epsilon}$ is set small enough, this problem will be feasible.
\end{remark}

Herein our goal is to formulate a neural network $\lambda: \wcal \to \R^M$ that is dual feasible, i.e., $\lambda(w) \in \Lambda(w)$ for every $w \in \wcal$.  
Let $\bar{\lambda}(w) \in \Lambda(w)$ define a \emph{strictly} feasible point for every $w \in \wcal$ as given in Assumption~\ref{ass:dual-feasible}.  
As discussed within Remark~\ref{rem:convex}\eqref{rem:convex-equality}, we want to consider a reformulation of $\Lambda(w)$ to eliminate the need for the equality constraints so that this set takes the general form of $\bbx$.
Specifically, let $(A^\T)_{\perp} \in \R^{M \times (M-N)}$ provide a basis for the null space of $A^\T$ (assuming $A^\T$ has full rank).  Then $\Lambda(w)$ can be reformulated without the equality constraint as
\[\Lambda(w) = \{(A^\T)_{\perp} z + \bar{\lambda}(w) \; | \; -[(A^\T)_{\perp} z - \bar{\lambda}(w)] \leq {\bf 0},~z \in \R^{M-N}\}.\]

Within the following proposition, we construct an analog of Proposition~\ref{prop:primal} for this dual feasibility constraint.  
\begin{proposition}\label{prop:dual}
Let $z: \wcal \to \R^{M-N}$ define an arbitrary neural network.  Let $\lambda: \wcal \to \R^M$ such that
\begin{align*}
\lambda(w) &:= (1-\bar{t}^*((A^\T)_\perp z(w) + \bar{\lambda}(w),\bar{\lambda}(w)))(A^\T)_\perp z(w) + \bar{y}(w),\\
\bar{t}^*(\hat{z},\bar{\lambda}) &:= \max_{j = 1,...,M} \left\{\frac{\hat{z}_j^-}{\hat{z}_j^- + \bar{\lambda}_j}\right\},
\end{align*}
where $\bar{t}^*(\hat{z},\bar{\lambda}) = 0$ if $\hat{z}_j \geq 0$ for every $j = 1,...,M$ and $\bar{\lambda}: \wcal \to \R^M$ provides a strictly feasible point for any $w \in \wcal$.  Then $\lambda(w) \in \Lambda(w)$ for every $w \in \wcal$.
\end{proposition}
Due to the similarity of this proposition to Proposition~\ref{prop:primal} we skip the proof of it here.

\subsubsection{Loss function}
Consider now the modified weighted-sum scalarization problem~\eqref{eq:scalarization-convex} with fixed weight $\delta \in (0,1)$.
An application of Theorem~\ref{thm:kkt} implies that $(\bar{x}^*(w,\delta),\bar{\lambda}^*(w,\delta)) \in \R^N \times \R^M$ are a pair of primal ($\bar{x}^*(w,\delta)$) and dual ($\bar{\lambda}^*(w,\delta)$) solutions of the scalarization problem~\eqref{eq:scalarization-convex}, respectively its Lagrange dual problem,  w.r.t.\ $w \in \wcal$ and $\delta \in (0,1)$ if and only if they jointly satisfy the KKT conditions
\begin{align*}
{\bf 0} &= (1-\delta)C^\T w + \delta \nabla\bar{f}(\bar{x}^*(w,\delta)) + A^\T \bar{\lambda}^*(w,\delta) \\
b &\geq A\bar{x}^*(w,\delta) \\ 
{\bf 0} &\leq \bar{\lambda}^*(w,\delta) \\
0 &= \bar{\lambda}_i^*(w,\delta) [A\bar{x}^*(w,\delta) - b]_i  \quad i=1,...,m.
\end{align*}
Note that the primal and dual feasibility constraints of the original, unmodified, scalarization problems~\eqref{eq:scalarization} are included in those of the modified scalarization problem~\eqref{eq:scalarization-convex}. Thus, the neural networks ($x(\cdot),\lambda(\cdot)$) constructed in Sections~\ref{subsec:p} and~\ref{subsec:d} from the original, unmodified, scalarization problem~\eqref{eq:scalarization} are automatically also feasible for the modified problem~\eqref{eq:scalarization-convex}. Hence, the only difference in methodology is to use the risk functional of the modified problem~\eqref{eq:scalarization-convex} instead of the original risk functional~\eqref{eq:loss} of problem~\eqref{eq:scalarization}.
The risk functional based on the KKT conditions of the modified problem~\eqref{eq:scalarization-convex} can be constructed in analogy to Section~\ref{sec:nn-loss} as
\begin{equation*}
L(x(\cdot),\lambda(\cdot);\eta) := \frac{1}{K} \sum_{k = 1}^K \left(\begin{array}{l} \|(1-\delta)C^\T w_k + \delta\nabla\bar{f}(x(w_k)) + A^\T \lambda(w_k)\|_2^2 \\ \quad + \eta \|\diag(\lambda(w_k)) [Ax(w_k) - b]\|_2^2\end{array}\right).
\end{equation*}
Recall from Proposition~\ref{prop:berge} that for $\delta \in (0,1)$ small enough, if the loss of the neural networks is small enough, we can confidently conclude that the obtained solutions $(\bar{x}^*(\cdot,\delta),\bar{\lambda}^*(\cdot,\delta))$ are ``close'' to the true solutions of~\eqref{eq:scalarization}.

\subsubsection{Inner and outer approximations}
Following the construction of the primal and dual neural networks for the LVOP, we want to comment on the inner and outer approximations as presented in Section~\ref{sec:approx}.
Formally, the approximation results of Section~\ref{sec:approx} follow directly from feasibility (both primal and dual feasibility) and strong duality.  
As we recover feasibility of our neural networks, and strong duality still holds due to Slater's condition, no modifications are necessary to recover these approximations.
In particular, the primal neural network $x: \wcal \to \R^N$ provides an inner approximation and the dual neural network $\lambda: \wcal \to \R^M$ provides an outer approximation of the upper set $\pcal$ via the relation
\[\cl\co\bigcup_{w \in \wcal}\left[C x(w) + \R^P_+\right] \subseteq \pcal \subseteq \bigcap_{w \in \wcal} \left\{y \in \R^P \; | \; w^\T y \geq -b^\T\lambda(w)\right\}.\]

\subsubsection{Discussion and numerical results}

Consider now the bi-objective ($P = 2$) LVOP with $N = 2$ decision variables and $M = 5$ inequality constraints constructed by
\[C := \left(\begin{array}{cc} 1 & 0 \\ 0 & 1 \end{array}\right), \qquad A := \left(\begin{array}{rr} -2 & -1 \\ -1 & -2 \\ 1 & 1 \\ -1 & 0 \\ 0 & -1 \end{array}\right), \quad\text{ and }\quad b := \left(\begin{array}{r} -2 \\ -2 \\ 6 \\ 0 \\ 0 \end{array}\right).\]
To place it within the strictly convex setting, we introduce the augmented form with $\bar f(x) := \|x\|_{2}^2$ included with constant weight $\delta = 10^{-4}$.

As in Section~\ref{sec:cs}, the neural networks for this example were computed with PyTorch on a local machine using the Adam optimizer with learning rate $10^{-4}$.  Here, we weight the complimentary slackness condition~\eqref{eq:kkt-cs} with $\eta = 10^{-4}$ so that the first-order condition error~\eqref{eq:kkt} and complimentary slackness error are, initially, of the same order of magnitude.  We train these neural networks for 500 epochs.  We wish to note that all other hyperparameters -- except the terminal activation functions as presented within this section -- are chosen arbitrarily and were not found via, e.g., a grid search.

Specifically, we have two neural networks to design: the primal ($x: \wcal \to \R^2$) and dual ($\lambda: \wcal \to \R^5$) neural networks.  The primal neural network $x$ is designed with three hidden layers, each with 800 hidden nodes; the terminal (projection) activation function is considered with a tolerance of $5 \times 10^{-5}$ to guarantee primal feasibility as discussed in Remark~\ref{rem:tolerance} and with strictly feasible point $\bar{x} := {\bf 1}$.  The dual neural network $\lambda$ is designed with three hidden layers, each with 800 hidden nodes.  The terminal (projection) activation function is considered with a tolerance of $5 \times 10^{-5}$ to guarantee dual feasibility as discussed in Remark~\ref{rem:tolerance} and with strictly feasible points $\bar{\lambda}: \wcal \to \R^M$ constructed via the linear programming approach outlined within Remark~\ref{rem:dual_strict_feasible} with $\bar{\epsilon} := 5 \times 10^{-3} \times {\bf 1}$.
All intermediate activation functions (i.e., all but the terminal activation functions) are chosen to be the hyperbolic tangent function with linear linking between layers.  

Figure~\ref{fig:lvop-efficient} displays the true weak efficient frontier along with the neural network inner and outer approximations on test scalarizations $w \in \{(i/500,1-i/500) \; | \; i \in \{0,1,...,500\}\}$ after training on $w \in \{(i/19 , 1-i/19) \; | \; i \in \{0,1,...,19\}\}$ with $\delta = 10^{-4}$.  The true weak efficient frontier is plotted with a solid black curve.  The inner approximation provided by the primal neural network (i.e., the boundary $\bd\cl\co\bigcup_{w \in \wcal} [Cx(w) + \R^2_+]$) is plotted as a solid blue line and the outer approximation provided by the dual neural network (i.e., the boundary $\bd\bigcap_{w \in \wcal} \{y \in \R^2 \; | \; w^\T y \geq -b^\T\lambda(w)\}$) is plotted as a solid red line.  Immediately noticeable, the neural network approximations almost completely overlap with each other and, in particular, capture the vertices of the true weak efficient frontier.  
The errors $\epsilon(w)$ can be directly seen in Figure~\ref{fig:2obj-epsilon} which displays the data in a logarithmic scale.  For these scalarizations, the original neural network approaches provide relatively high errors (see the blue line ``NN Error'').  However, when using the realizations of the neural network approximation -- the convex hull in the inner approximation and the intersection in the outer approximation as provided in Lemma~\ref{lemma:approx} -- these errors drop significantly with most scalarizations having errors around $10^{-2}$.
Further, we wish to remind the reader that the neural networks were trained without hyperparameter tuning and, as such, the errors $\epsilon(w)$ for the machine learning approach can be further improved.

\begin{figure}
\centering
\begin{subfigure}[t]{0.45\textwidth}
\centering
\includegraphics[width=\textwidth]{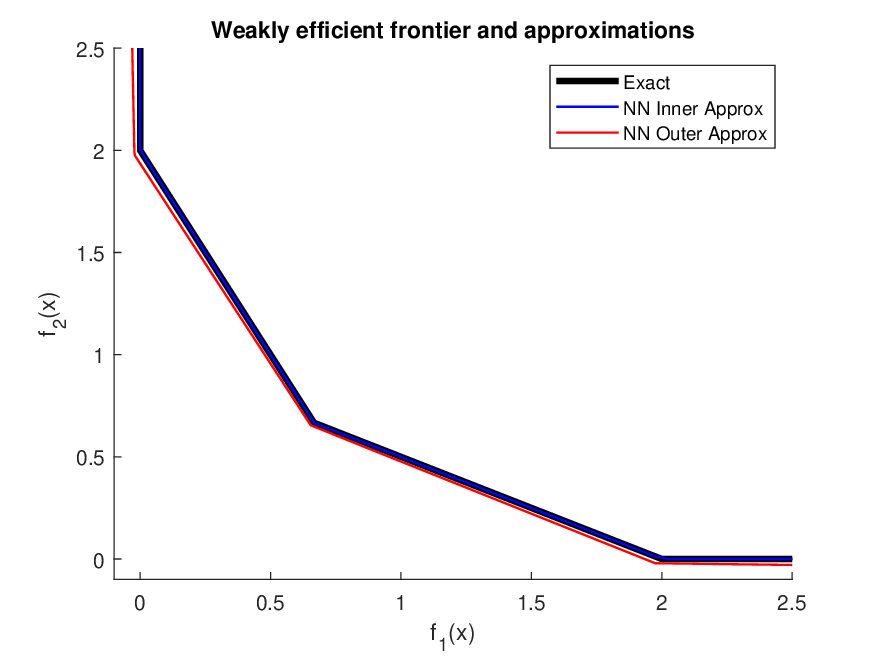}
\caption{Plot of weak efficient frontier (black) and approximations through neural networks (blue and red).} 
\label{fig:lvop-efficient}
\end{subfigure}
~
\begin{subfigure}[t]{0.45\textwidth}
\centering
\includegraphics[width=\textwidth]{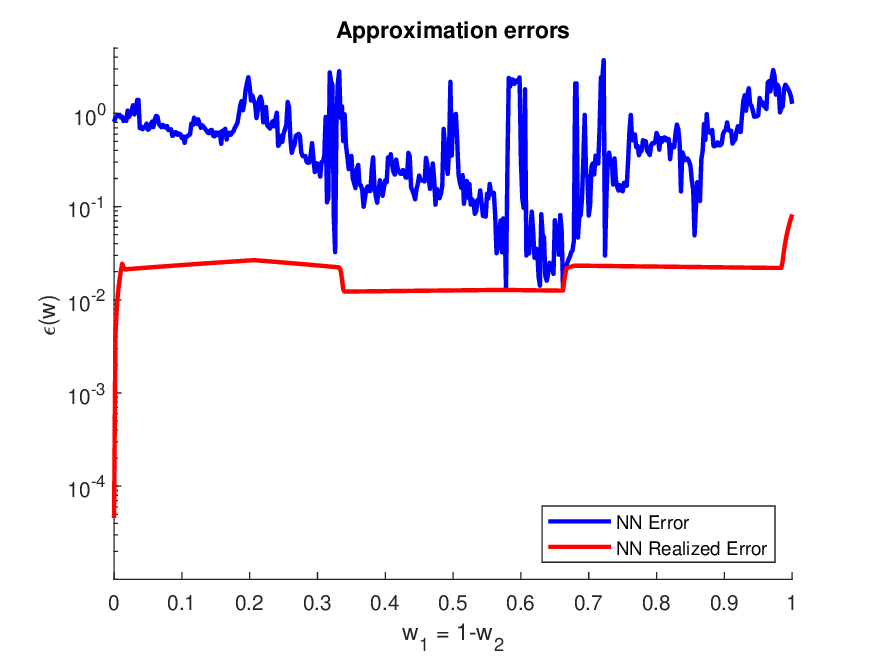}
\caption{Plot of approximation errors $\epsilon(\cdot)$ in log-scale so that the neural network (blue) and realized approximations (red) provide $\epsilon(\cdot)$-inner approximations.}
\label{fig:lvop-epsilon}
\end{subfigure}
\caption{Section~\ref{sec:convex-lvop}: Plot of the efficient frontier and approximation errors for 501 regularly spaced test scalarizations. Neural networks are computed with the 20 regularly spaced training scalarizations.}
\label{fig:lvop}
\end{figure}

\subsection{Mean-risk problem}\label{sec:mr}
Herein we want to revisit a practical example by studying the mean-risk problem in finance with dynamic trading.  Similar to the mean-variance problem presented within Section~\ref{sec:mean-var}, this is a bi-objective problem that allows for easy visualization of the optimal portfolio.  However, in contrast to the mean-variance problem, herein we consider risk as measured by the entropic risk measure \cite[Example 4.34]{FS11}.  We will present this problem with no-short selling allowed, i.e., the investor can only purchase assets for her portfolio.  To simplify this setting we will consider a market with $S = 2$ assets only: asset $0$ is the cash asset and is worth $\$1$ throughout; asset $1$ is a risky stock whose value fluctuates (with equal up or down probability) according to the Cox-Ross-Rubinstein binomial tree~\cite{cox1979option} with $T = 2$ time steps and volatility $\sigma^2 = 0.05$ as depicted in Figure~\ref{fig:mr-tree}.

Within this setting, the investor is optimizing over $6$ variables -- investments in each asset at time $0$ and in each asset in either state at time $1$.  
Let $\hat x = (\hat x_{0,0} , \hat x_{0,1} , \hat x_{1,0}^U , \hat x_{1,1}^U , \hat x_{1,0}^D , \hat x_{1,1}^D) \in \R^6$ provide the investments at times $0,1$ and in assets $0,1$.  The investor is seeking to maximize her expected return and minimize her entropic risk at time $T = 2$, i.e.,
\begin{align*}
f_1(\hat x) &= -\frac{1}{2}\left[\left(\hat x_{1,0}^U + \frac{S_2^{UU} + S_2^{UD}}{2}\hat x_{1,1}^U\right) + \left(\hat x_{1,0}^D + \frac{S_2^{UD} + S_2^{DD}}{2}\hat x_{1,1}^D\right)\right] \\
f_2(\hat x) &= \frac{1}{\alpha}\log\left(\frac{1}{4}\left[e^{-\alpha[\hat x_{1,0}^U + S_2^{UU}\hat x_{1,1}^U]} + e^{-\alpha[\hat x_{1,0}^U + S_2^{UD}\hat x_{1,1}^U]} + e^{-\alpha[\hat x_{1,0}^D + S_2^{UD}\hat x_{1,1}^D]} + e^{-\alpha[\hat x_{1,0}^D + S_2^{DD}\hat x_{1,1}^D]}\right]\right)
\end{align*}
for risk-aversion parameter $\alpha > 0$.  Herein we will consider $\alpha = 1/2$ chosen arbitrarily.  
As this problem is convex, but not strictly convex, we augment this problem with a 3rd objective function: $f_3(\hat x) = \|\hat x\|_2^2$ which is always included with weight $\delta = 10^{-4}$.

Due to the no-short selling constraint, the investor is constrained by $M = S$ inequality constraints $g(\hat x) = -\hat x \in \R^6$.
In addition to these inequality constraints, this constructed dynamic portfolio is required to satisfy self-financing conditions with initial wealth $W = 10$ such that
\begin{align*}
W &= \hat x_{0,0} + S_0 \hat x_{0,1} \\
0 &= \left(\hat x_{0,0} + S_1^U \hat x_{0,1}\right) - \left(\hat x_{1,0}^U + S_1^U \hat x_{1,1}^U\right) \\
0 &= \left(\hat x_{0,0} + S_1^D \hat x_{0,1}\right) - \left(\hat x_{1,0}^D + S_1^D \hat x_{1,1}^D\right).
\end{align*}
As discussed within Sections~\ref{sec:background-vop} and~\ref{sec:mean-var}, we consider the modification of this problem to guarantee these equality constraints are satisfied.  Denote these self-financing constraints by $A\hat x = (W,0,0)^\T$; then we consider $N = 3$ variables with $\hat{x} := A_{\perp} x + \tilde x$ for $x \in \R^N$ with $A_{\perp} \in \R^{6 \times 3}$ and $\tilde x := \frac{W}{2}{\bf 1} \in \R^6$.

To consider the outer approximation, we need to consider the Lagrange dual function $d: \R^6 \times \wcal \to \R$ of the weighted-sum scalarizations (prior to augmentation).  For this mean-risk problem, the Lagrange dual function does not have a closed form solution, rather it is computed by solving the unconstrained problem
\begin{align*}
d(\lambda,w) := \inf_{x \in \R^3} \left[w^\T f(A_{\perp} x + \tilde x) - \lambda^\T(A_{\perp} x + \tilde x)\right]
\end{align*}
for any $\lambda \in \R^3$ and $w \in \wcal$.

With this setting, we have two neural networks to design: the primal ($x: \wcal \to \R^3$) and dual ($\lambda: \wcal \to \R^6$) neural networks.  
The primal neural network $x$ is designed with three hidden layers, each with 800 hidden nodes; the terminal (projection) activation function is considered with a tolerance of $5 \times 10^{-5}$ to guarantee primal feasibility as discussed in Remark~\ref{rem:tolerance} and with strictly feasible point $\bar{x} := {\bf 0} \in \R^3$.  The dual neural network $\lambda$ is designed with three hidden layers, each with 800 hidden nodes.  The terminal ReLU activation function is used to guarantee dual feasibility.
All intermediate activation functions (i.e., all but the terminal activation functions) are chosen to be the hyperbolic tangent function with linear linking between layers.  
As with the prior examples, the neural networks herein are computed with PyTorch on a local machine using the Adam optimizer with learning rate $10^{-4}$.  Here, we weight the complimentary slackness condition~\eqref{eq:kkt-cs} with $\eta = 10$.  We train these neural networks for $2000$ epochs over 4 training points $\{(0,1)^\T , (\frac{1}{3},\frac{2}{3})^\T , (\frac{2}{3},\frac{1}{3})^\T , (1,0)^\T\}$.  We wish to note that all other hyperparameters -- except the terminal activation functions discussed previously -- are chosen arbitrarily and were not found via, e.g., a grid search.

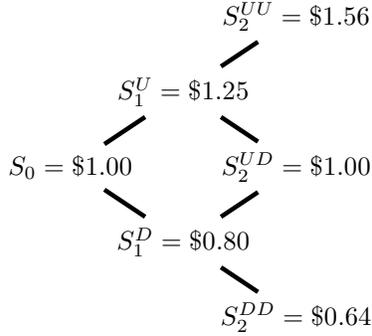
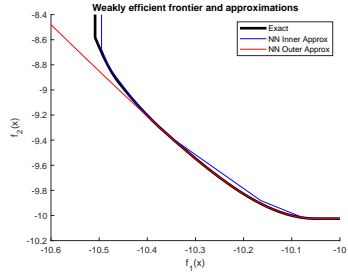
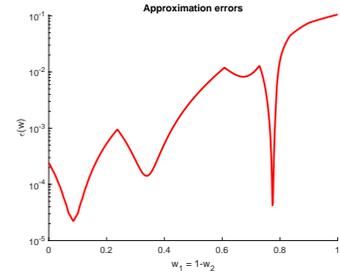
\begin{figure}
\centering
\begin{subfigure}[t]{0.3\textwidth}
\centering
\begin{tikzpicture}
\node at (0cm,0cm) (n0) {\footnotesize $S_0 = \$1.00$};
\node at (1.5cm,1cm) (nU) {\footnotesize $S_1^U = \$1.25$};
\node at (1.5cm,-1cm) (nD) {\footnotesize $S_1^D = \$0.80$};
\node at (3cm,2cm) (nUU) {\footnotesize $S_2^{UU} = \$1.56$};
\node at (3cm,0cm) (nUD) {\footnotesize $S_2^{UD} = \$1.00$};
\node at (3cm,-2cm) (nDD) {\footnotesize $S_2^{DD} = \$0.64$};

\draw[-,line width=.6mm] (n0) -- (nU);
\draw[-,line width=.6mm] (n0) -- (nD);
\draw[-,line width=.6mm] (nU) -- (nUU);
\draw[-,line width=.6mm] (nU) -- (nUD);
\draw[-,line width=.6mm] (nD) -- (nUD);
\draw[-,line width=.6mm] (nD) -- (nDD);
\end{tikzpicture}
\caption{Depiction of the price process as a tree. All transitions occur with probability $50\%$.}
\label{fig:mr-tree}
\end{subfigure}
~
\begin{subfigure}[t]{0.3\textwidth}
\centering
\includegraphics[width=\textwidth]{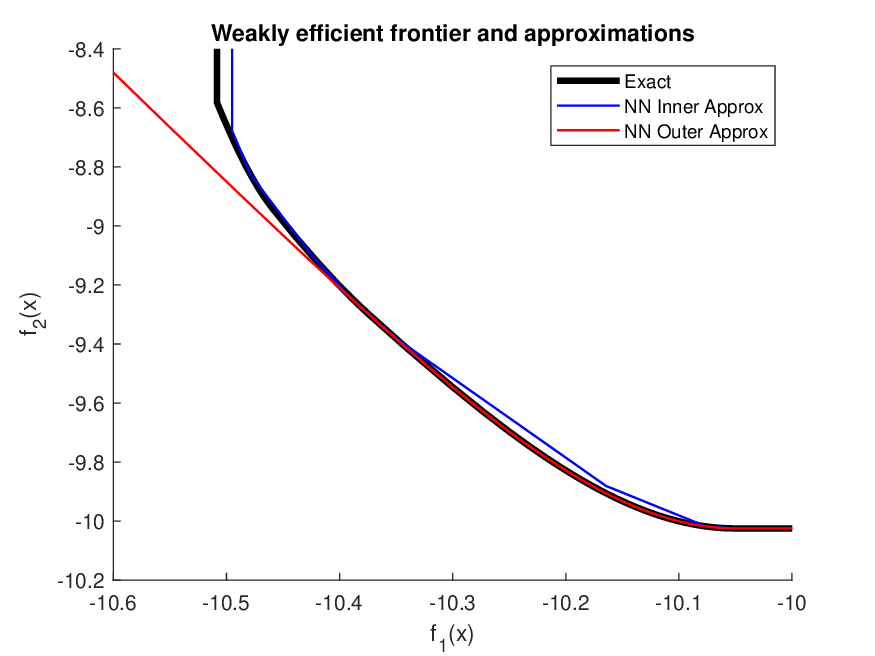}
\caption{Plot of weak efficient frontier (black) and approximations through neural networks (blue and red).} 
\label{fig:mr-efficient}
\end{subfigure}
~
\begin{subfigure}[t]{0.3\textwidth}
\centering
\includegraphics[width=\textwidth]{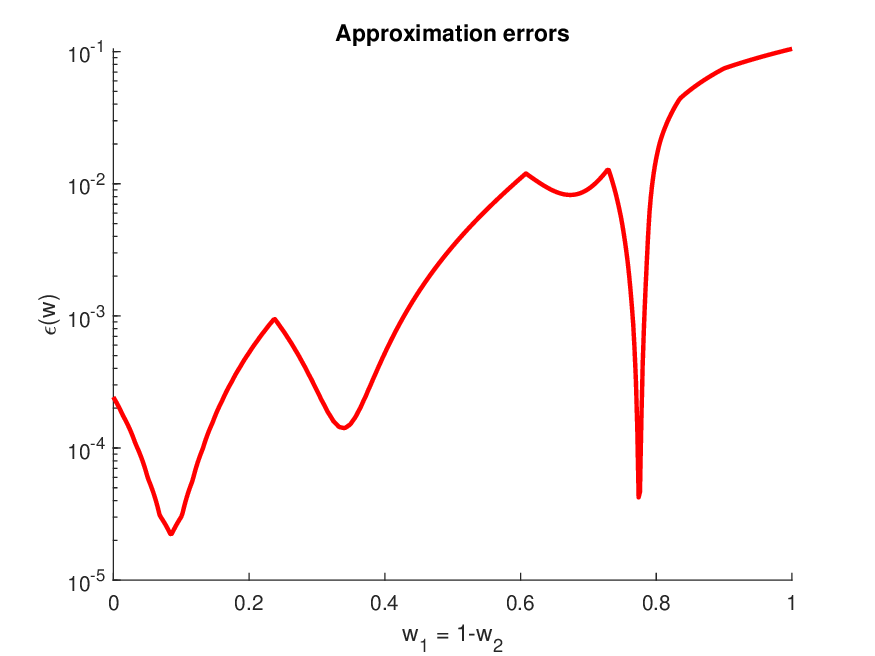}
\caption{Plot of the realized approximation errors $\epsilon(\cdot)$.}
\label{fig:mr-epsilon}
\end{subfigure}
\caption{Section~\ref{sec:mr}: Plot of the market model, efficient frontier, and approximation errors on 501 regularly spaced test scalarizations.  Neural networks are computed with 4 regularly spaced training scalarizations.}
\label{fig:mr}
\end{figure}

Figure~\ref{fig:mr-efficient} displays the true weak efficient frontier along with the neural network inner and outer approximations on 501 regularly spaced test scalarizations $w \in \{(i/500,1-i/500) \; | \; i \in \{0,1,...,500\}\}$.  The true weak efficient frontier is plotted with a solid black curve.  The inner approximation provided by the primal neural network (i.e., the boundary $\bd\cl\co\bigcup_{w \in \wcal}[f(x(w)) + \R^2_+]$) is plotted as a solid blue line and the outer approximation provided by the dual neural network (i.e., the boundary $\bd\bigcap_{w \in \wcal}\{y \in \R^2 \; | \; w^\T y \geq d(\lambda(w),w)\}$) is plotted as a solid red line.  In this scenario, the dual neural network uniformly provide the zero output, i.e., $\lambda \equiv {\bf 0}$.  We conjecture this occurs because the true dual value $\lambda^*(w)$ is nearly identically 0 throughout, and only deviates when $w_1 \approx 1$ or $w_2 \approx 0$ with $\lambda^*((1,0)) \approx (0.0257 \, , \, 0 \, , \, 0.0126 \, , \, 0 \, , \, 0.0126 \, , \, 0)^\T$.  This can be observed with the strong fit of the outer approximation except in the upper left corner of Figure~\ref{fig:mr-efficient}.  The strength of these approximations can further be seen in Figure~\ref{fig:mr-epsilon} where the errors $\epsilon(w)$ are plotted directly in a logarithmic scale.  With an improvement in the dual neural network to capture behavior for high values of $w_1$, we would expect these approximation errors to remain on the order of $10^{-3}$ or below. We again wish to remind the reader that the neural network architectures considered herein were chosen arbitrarily and not as the result of hyperparameter tuning.

\section{Conclusion}\label{sec:conclusion}
Within this work we proposed a neural network architecture that provides inner and outer approximations of the weakly efficient frontier for CVOPs. Through the use of the dual of the scalarization, we provided a functional of the primal and dual neural networks which provides an upper bound to the error of the inner approximation at each point on the weakly efficient frontier. To demonstrate the ability of the proposed method for large-scale problems, we computed the approximations of the efficient frontiers for multiple CVOPs.

We wish to conclude with an extension to this methodology which we believe would be of great interest.  Within this work, only the final activation function was investigated.  We did not explore the impact of different intermediate layer and activation function structures.  Optimizing this structure can reduce the size of the realized errors and improve the approximations.  We propose investigating such a hyperparameter optimization for specific classes of problems, e.g., for quadratic vector optimization as analytical structures are most likely to be tractable.

\section{Acknowledgement}
We would like to thank Gabriele Eichfelder for her very helpful comments on an earlier version of this paper.

\bibliographystyle{plain}
\bibliography{biblio}

\end{document}